\documentclass{amsart}

\usepackage{epsfig}
\usepackage{amsmath}
\usepackage{amssymb,mathrsfs}
\usepackage{footmisc} 

\usepackage[all]{xy}

\usepackage{color}
\usepackage{tabularx}

\def\comment#1{{\sf{[#1]}}}

\def\Z{{\mathbb Z}}
\def\Q{{\mathbb Q}}

\def\R{{\mathbb R}}
\def\C{{\mathbb C}}
\def\P{{\mathbb P}}
\def\F{{\mathbb F}}

\def\L{{\mathbb L}}

\def\kk{{\Bbbk}}		

\def\cA{{\mathcal A}}
\def\cB{{\mathcal B}}

\def\cG{{\mathcal G}}

\def\cJ{{\mathcal J}}
\def\K{{\mathcal K}}

\def\M{{\mathcal M}}

\def\cO{{\mathcal O}}

\def\U{{\mathcal U}}

\def\X{{\mathcal X}}

\def\sE{{\mathscr E}}

\def\sL{{\mathscr L}}
\def\sP{{\mathscr P}}
\def\sS{{\mathscr S}}
\def\sV{{\mathscr V}}

\def\k{{\mathfrak k}}
\def\p{{\mathfrak p}}

\def\u{{\mathfrak u}}

\def\e{{\epsilon}}
\def\w{{\omega}}
\def\G{{\Gamma}}

\def\bL{\boldsymbol{L}} 
\def\bP{\boldsymbol{P}}
\def\bV{\boldsymbol{V}}

\def\alphahat{{\hat{\alpha}}}

\def\chitilde{{\tilde{\chi}}}

\def\rhotilde{{\tilde{\rho}}}

\def\tautilde{{\tilde{\tau}}}

\def\cGhat{{\widehat{\cG}}}
\def\Mhat{{\widehat{M}}}

\def\Uhat{{\widehat{\U}}}
\def\Xhat{{\widehat{X}}}

\def\uhat{{\widehat{\u}}}

\def\betadot{{\dot{\beta}}}
\def\gammadot{{\dot{\gamma}}}

\def\cGbar{\overline{\cG}}

\def\Mbar{{\overline{\M}}}

\def\Qbar{{\overline{\Q}}}
\def\Sbar{{\overline{S}}}
\def\Ubar{{\overline{\U}}}
\def\Xbar{{\overline{X}}}
\def\Ybar{{\overline{Y}}}

\def\Deltabar{{\overline{\Delta}}}

\def\Pdual{\check{\bP}}
\def\Ldual{\check{\bL}}

\def\Hdual{\check{H}}
\def\deltadual{\check{\delta}}

\def\vv{{\vec{\mathsf v}}}

\def\qq{{\vec{\mathsf q}}}

\def\MHS{{\mathsf{MHS}}}
\def\MTM{{\mathsf{MTM}}}

\def\GRT{{\mathrm{GRT}}}
\def\URT{{\mathrm{URT}}}

\def\et{{\mathrm{\acute{e}t}}}
\def\cts{{\mathrm{cts}}}
\def\un{{\mathrm{un}}}
\def\ss{{\mathrm{ss}}}

\def\op{{\mathrm{op}}}
\def\rot{{\mathrm{rot}}}

\def\To{\longrightarrow}
\def\bdot{{\bullet}}

\def\blank{{\underline{\phantom{x}}}}		
\def\dd{{\mathbf{d}}}

\def\KV{{\mathrm{KV}}}
\def\KVI{{\mathrm{KVI}}}

\def\cs{\beta_{C\!S}}

\def\csdual{{\check{\beta}_{C\!S}}}
\def\phidual{{\check{\varphi}}}
\def\psidual{{\check{\psi}}}
\def\partialdual{{\check{\partial}}}

\def\bil{{\langle\phantom{x},\phantom{x}\rangle}}
\def\gold{\{\phantom{x},\phantom{x}\}}

\def\ll{\langle\langle}
\def\rr{\rangle\rangle}

\def\Pminus{{\P^1-\{0,1,\infty\}}}
\def\Ql{{\Q_\ell}}
\def\Zl{{\Z_\ell}}

\def\uu{{\vec{1}}}

\def\MT{{\mathrm{MT}}}
\def\Gm{{\mathbb{G}_m}}
\def\Sp{{\mathrm{Sp}}}
\def\GSp{{\mathrm{GSp}}}
\def\SL{{\mathrm{SL}}}

\def\KRV{{\mathcal{KRV}}}						

\newcommand\id{\operatorname{id}}

\newcommand\Hom{\operatorname{Hom}}
\newcommand\Ext{\operatorname{Ext}}

\newcommand\Aut{\operatorname{Aut}}
\newcommand\Diff{\operatorname{Diff}}

\newcommand\Der{\operatorname{Der}}
\newcommand\Gr{\operatorname{Gr}}

\newcommand\cone{\operatorname{cone}}

\newcommand\Gal{\operatorname{Gal}}

\newcommand\Jac{\operatorname{Jac}}

\newcommand\Spec{\operatorname{Spec}}
\newcommand\cupdot{\operatorname{\dot{\cup}}}
\newcommand\Arf{\operatorname{Arf}}

\newcommand\comptensor{\operatorname{\widehat{\otimes}}}

\numberwithin{equation}{section}

\newtheorem{theorem}{Theorem}[section]
\newtheorem{lemma}[theorem]{Lemma}
\newtheorem{proposition}[theorem]{Proposition}
\newtheorem{corollary}[theorem]{Corollary}
\newtheorem{bigtheorem}{Theorem}
\newtheorem{bigcorollary}[bigtheorem]{Corollary}

\theoremstyle{definition}
\newtheorem{definition}[theorem]{Definition}

\newtheorem{hypothesis}[theorem]{Hypothesis}

\theoremstyle{remark}
\newtheorem{remark}[theorem]{Remark}

\newtheorem{conjecture}[theorem]{Conjecture}

\begin{document}

\title[Hodge Theory of the Turaev Cobracket]{Hodge Theory of the Turaev Cobracket and the Kashiwara--Vergne Problem}

\author{Richard Hain}
\address{Department of Mathematics\\ Duke University\\
Durham, NC 27708-0320}
\email{hain@math.duke.edu}

\thanks{Supported in part by grant DMS-1406420 from the National Science Foundation.}
\thanks{ORCID iD: {\sf orcid.org/0000-0002-7009-6971}}

\date{\today}

\subjclass{Primary 	17B62, 58A12; Secondary 57N05, 14C30}

\keywords{Turaev cobracket, Goldman bracket, Lie bialgebra, Hodge theory}

\begin{abstract} 
In this paper we show that, after completing in the $I$-adic topology, the Turaev cobracket on the vector space freely generated by the closed geodesics on a smooth, complex algebraic curve $X$ with a quasi-algebraic framing is a morphism of mixed Hodge structure. We combine this with results of a previous paper on the Goldman bracket to construct torsors of solutions to the Kashiwara--Vergne problem in all genera. The solutions so constructed form a torsor under a prounipotent group that depends only on the topology of the framed surface. We give a partial presentation of these groups. Along the way, we give a homological description of the Turaev cobracket.
\end{abstract}

\maketitle


\section{Introduction}

Denote the set of free homotopy classes of maps $S^1 \to X$ in a topological space $X$ by $\lambda(X)$ and the free $R$-module it generates by $R\lambda(X)$. When $X$ is an oriented surface with a nowhere vanishing vector field $\xi$, there is a map
$$
\delta_\xi : R\lambda(X) \to R\lambda(X)\otimes R\lambda(X),
$$
called the {\em Turaev cobracket}, that gives $R\lambda(X)$ the structure of a Lie coalgebra. The cobracket was first defined by Turaev \cite{turaev:78} on $R\lambda(M)/R$ (with no framing) and lifted to $R\lambda(M)$ for framed surfaces in \cite[\S18]{turaev:skein} and \cite{akkn2}. The cobracket $\delta_\xi$ and the Goldman bracket \cite{goldman}
$$
\gold : R\lambda(X)\otimes R\lambda(X) \to R\lambda(X)
$$
endow $R\lambda(X)$ with the structure of an involutive Lie bialgebra \cite{turaev:skein,chas,kk:intersections}.

The value of the cobracket on a loop $a \in \lambda(X)$ is obtained by representing it by an immersed circle $\alpha : S^1 \to X$ with transverse self intersections and trivial winding number relative to $\xi$. Each double point $P$ of $\alpha$ divides it into two loops based at $P$, which we denote by $\alpha'_P$ and $\alpha_P''$. Let $\e_P = \pm 1$ be the intersection number of the initial arcs of $\alpha_P'$ and $\alpha_P''$. The cobracket of $a$ is then defined by
\begin{equation}
\label{eqn:def_turaev}
\delta_\xi(a) = \sum_P \e_P(a'_P \otimes a''_P - a''_P \otimes a'_P),
\end{equation}
where $a_P'$ and $a_P''$ are the classes of $\alpha_P'$ and $\alpha_P''$, respectively.

The powers of the augmentation ideal $I$ of $R\pi_1(X,x)$ define the $I$-adic topology on it and induce a topology on $R\lambda(X)$. Kawazumi and Kuno \cite{kk:intersections} showed that $\delta_\xi$ is continuous in the $I$-adic topology and thus induces a map
$$
\delta_\xi : R\lambda(X)^\wedge \to R\lambda(X)^\wedge\comptensor R\lambda(X)^\wedge
$$
on $I$-adic completions. This and the completed Goldman bracket give $R\lambda(X)^\wedge$ the structure of an involutive completed Lie bialgebra \cite{kk:intersections}.

Now suppose that $X$ is a smooth affine curve over $\C$ or, equivalently, the complement of a non-empty finite set $D$ in a compact Riemann surface $\Xbar$. In this case $\Q\lambda(X)^\wedge$ has a canonical pro-mixed Hodge structure \cite{hain:dht}. In particular, it has a {\em weight filtration}
$$
\cdots \subseteq W_{-2} \Q\lambda(X)^\wedge \subseteq W_{-1} \Q\lambda(X)^\wedge \subseteq W_0 \Q\lambda(X)^\wedge = \Q\lambda(X)^\wedge
$$
and its complexification $\C\lambda(X)^\wedge$ has a {\em Hodge filtration}
$$
\cdots \supset  F^{-2} \C\lambda(X)^\wedge \supset F^{-1} \C\lambda(X)^\wedge \supset F^{0} \C\lambda(X)^\wedge \supset F^1 \C\lambda(X)^\wedge = 0.
$$
The Hodge filtration depends on the algebraic structure on $X$, while the weight filtration is topologically determined and so does not depend on the complex structure.\footnote{The weight filtration on $\Q\lambda(X)^\wedge$ is the image of the weight filtration of $\Q\pi_1(X,x)^\wedge$, which is determined uniquely by the conditions that $W_{-1} \Q\pi_1(X,x)^\wedge  = I$, $W_{-2} \Q\pi_1(X,x)^\wedge=\ker\{I\to H_1(\Xbar)\}$, and by the condition that $W_{-m-2}\Q\pi_1(X,x)^\wedge$ is the ideal generated by $W_{-1}W_{-m-1}$ and $W_{-2}W_{-m}$.} This pro-mixed Hodge structure contains subtle geometric and arithmetic information about $X$. Together with the image of $\Z\lambda(X)^\wedge$, the mixed Hodge structure on $\Q\lambda(X)/I^3$ determines, for example, $\Xbar$ up to isomorphism, $X$ up to finite ambiguity, as well as the point in the intermediate jacobian of the jacobian of $\Xbar$ determined by the Ceresa cycle of $\Xbar$.

Our first main result is that the Turaev cobracket is compatible with this structure.

\begin{bigtheorem}
\label{thm:turaev}
If $\xi$ is a nowhere vanishing holomorphic vector field on $X$ that is meromorphic on $\Xbar$, then
$$
\delta_\xi : \Q\lambda(X)^\wedge\otimes\Q(-1) \to \Q\lambda(X)^\wedge\comptensor \Q\lambda(X)^\wedge
$$
is a morphism of pro-mixed Hodge structures, so that $\Q\lambda(X)^\wedge\otimes\Q(1)$ is a complete Lie coalgebra in the category of pro-mixed Hodge structures.
\end{bigtheorem}

We call such a framing $\xi$ an {\em algebraic framing}. The previous result also holds in the more general situation where the framing $\xi$ is a section of a twist of the holomorphic tangent bundle of $\Xbar$ by a torsion line bundle. We call such framings {\em quasi-algebraic framings} of $X$. (See Definition~\ref{def:q-alg}.)

The main result of \cite{hain:goldman} asserts that
$$
\gold : \Q\lambda(X)^\wedge\otimes\Q\lambda(X)^\wedge \to \Q\lambda(X)^\wedge \otimes \Q(1)
$$
is a morphism of mixed Hodge structure (MHS), so that $\Q\lambda(X)^\wedge\otimes\Q(-1)$ is a complete Lie algebra in the category of pro-mixed Hodge structures.

\begin{bigcorollary}
\label{cor:gr_gt}
If $\xi$ is a quasi-algebraic framing of $X$, then $\big(\Q\lambda(X)^\wedge,\gold,\delta_\xi\big)$ is a ``twisted'' completed Lie bialgebra in the category of pro-mixed Hodge structures.
\end{bigcorollary}

By ``twisted'' we mean that one has to twist both the bracket and cobracket by $\Q(\pm 1)$ to make them morphisms of MHS. There is no one twist of $\Q\lambda(X)$ that makes them simultaneously morphisms of MHS.

Let $\vv$ be a non-zero tangent vector of $\Xbar$ at a point of $D$. Standard results in Hodge theory (see \cite[\S10.2]{hain:goldman}) imply:

\begin{bigcorollary}
\label{cor:formality}
Hodge theory determines torsors of compatible isomorphisms
\begin{equation}
\label{eqn:split_GT}
\big(\Q\lambda(X)^\wedge,\gold,\delta_\xi\big) \overset{\simeq}{\To} \Big(\prod_{m\ge 0}\Gr^W_{-m}\Q\lambda(X)^\wedge,\Gr^W_\bdot\gold,\Gr^W_\bdot\delta_\xi\Big)
\end{equation}
of the Goldman--Turaev Lie bialgebra with the associated weight graded Lie bialgebra and of the completed Hopf algebras
\begin{equation}
\label{eqn:split_A}
\Q\pi_1(X,\vv)^\wedge \overset{\simeq}{\To} \prod_{m\ge 0}\Gr^W_{-m} \Q\pi_1(X,\vv)^\wedge
\end{equation}
under which the logarithm of the boundary circle lies in $\Gr^W_{-2} \Q\pi_1(X,\vv)^\wedge$. These isomorphisms are torsors under the prounipotent radical $U^\MT_{X,\vv}$ of the Mumford--Tate group of the MHS on $\Q\pi_1(X,\vv)^\wedge$.
\end{bigcorollary}
In the terminology of \cite{akkn2}, such isomorphisms solve the {\em Goldman--Turaev formality} problem.

There are many potential applications of these results to the study of the geometry and arithmetic of algebraic curves. In this paper, we will focus on an application to the Kashiwara--Vergne problem \cite{akkn2}, a problem in Lie theory related to Poisson geometry and the study of associators. Additional applications can be found in \cite{hain:johnson}.

Solutions of the Kashiwara--Vergne problem of type $(g,n+1)$, where $2g-1+n>0$, are automorphisms $\Phi$ of the complete Hopf algebra
$$
\Q\ll x_1,\dots,x_g,y_1,\dots,y_g,z_1,\dots,z_n\rr
$$
that solve the Kashiwara--Vergne equations. They correspond to automorphisms $\Phi$ that induce isomorphisms of the Goldman--Turaev Lie bialgebra with the completion of its associated weight graded that satisfy certain natural boundary conditions. Corollary~\ref{cor:formality}, combined with \cite[Thm.~5]{akkn1}, implies that the automorphism $\Phi$ constructed from a Hodge splitting of $\Q\pi_1(X,\vv)^\wedge$ in \cite[\S13.4]{hain:goldman} solves the KV equations. The following result is a special case of Corollary~\ref{cor:kv}.

\begin{bigcorollary}
\label{cor:big_kv}
Suppose that $X$ is an affine curve of type $(g,n+1)$, where $2g-1+n>0$. If $\xi$ is a quasi-algebraic framing of $X$, then the isomorphisms $\Phi$ constructed in \cite[\S13.4]{hain:goldman} from the canonical MHS on $\Q\pi_1(X,\vv)^\wedge$ are solutions of the Kashiwara--Vergne problem. The solutions constructed in this manner form a torsor under the unipotent radical $\U^\MT_{X,\vv}$ of the Mumford--Tate group of the canonical mixed Hodge structure on $\Q\pi_1(X,\vv)^\wedge$.
\end{bigcorollary}

Our solutions of the Kashiwara--Vergne problem have the property that the corresponding splitting of the filtrations are compatible with those of the Lie algebra of the relative completion of the mapping class group constructed in \cite{hain:torelli}. (See \cite[Thm.~6]{hain:goldman}.) Whether or not all solutions of the Kashiwara--Vergne problem have this property is not known.

The Kashiwara--Vergne problem concerns smooth surfaces and does not require a complex structure. Let $\Sbar$ be a closed oriented surface of genus $g$ and $P=\{x_0,\dots,x_n\}$ a finite subset. Set $S=\Sbar-P$.
Assume that $S$ is hyperbolic; that is, $2g-1+n>0$. Suppose that $\xi_o$ is a framing of $S$. Denote the index (or local degree) of $\xi_o$ at $x_j$ by $d_j$. Let $\dd = (d_0,\dots,d_n) \in \Z^{n+1}$ be the vector of local degrees of $\xi_o$. The Poincar\'e--Hopf Theorem implies that $\sum d_j = 2-2g$.

In \cite{akkn2} it is shown that the Kashiwara--Vergne problem admits solutions for all framed surfaces of genus $g\neq 1$ and for surfaces of genus 1 with certain, but not all, framings.\label{page:foot}\footnote{To compare the two statements, one should note that if $\gamma_j$ is the boundary of sufficiently small disk in $\Xbar$, centered at $x_j$, then $d_j +\rot_{\xi} \gamma_j = 1$. Note that the boundary orientation conventions used in \cite{akkn1,akkn,akkn2} differ from those used in this paper. Their ``adapted framing'' has the property that $d_0 = 2-2g$ and $d_j = 0$ for all $j\ge 1$.\label{foot:rot}} (See \cite[Thm.~6.1]{akkn2}.) We obtain an independent proof of their result by showing (in Section~\ref{sec:alg_framing}) that the framings for which the KV-problem has a solution are precisely those that can be realized by a quasi-algebraic framing. The proof combines work of Kawazumi \cite{kawazumi:framings} with the existence of meromorphic quadratic differentials, which is established in the works of Kontsevich--Zorich \cite{kontsevich-zorich} and Bainbridge, Chen, Gendron, Grushevsky and M\"oller \cite{mero_diffs}.

\begin{bigtheorem}
\label{thm:quasi-alg}
 If $g\neq 1$, then there is a complex structure $(\Xbar,D)$ on $(\Sbar,P)$ such that $\xi_o$ is homotopic to a quasi-algebraic framing of $X$. When $g=1$, then there is a complex structure on $(\Sbar,P)$ for which $\xi_o$ is quasi-algebraic if and only if the rotation number of $\rot_{\xi_o}\gamma$ of every simple closed curve $\gamma$ in $\Xbar$ is divisible by $\gcd(d_0,\dots,d_n)$.
\end{bigtheorem}

Solutions of the Kashiwara--Vergne (KV) problem for $(S,\xi_o)$ form a torsor under a prounipotent group, denoted $\KRV_{g,n+\uu}^\dd$ in \cite{akkn2}. It depends only on the vector of local degrees $\dd$ and not on other topological invariants of $\xi_o$. Corollary~\ref{cor:big_kv} implies that each quasi-algebraic structure
$$
\phi : (\Xbar,D,\vv,\xi) \overset{\simeq}{\To} (\Sbar,P,\vv_o,\xi_o)
$$
determines an injection $\U^\MT_{X,\vv} \hookrightarrow \KRV^\dd_{g,n+\uu}$. Letting the stabilizer of $\xi_o$ in the mapping class group of $(\Sbar,P)$ act on the complex structure $\phi$ by precomposition, we obtain a larger a larger torsor of solutions to the KV-problem. These form a torsor under a prounipotent group $\Uhat^\dd_{g,n+\uu}$ whose construction and structure is discussed below. It is a subgroup of $\KRV_{g,n+\uu}^\dd$. We conjecture that it is equal to $\KRV_{g,n+\uu}^\dd$. Equivalently, we conjecture that all solutions of the Kashiwara--Vergne problem arise from the Hodge theoretic constructions for some quasi-algebraic structure $\phi$.

In order to state the next theorem, we need to introduce several prounipotent groups. Denote the category of mixed Tate motives unramified over $\Z$ by $\MTM(\Z)$. Denote the prounipotent radical of its tannakian fundamental group $\pi_1(\MTM,\w^B)$ (with respect to the Betti realization $\w^B$) by $\K$. Its Lie algebra $\k$ is non-canonically isomorphic to the free Lie algebra
$$
\k \cong \L(\sigma_3,\sigma_5,\sigma_7,\sigma_9,\dots)^\wedge.
$$

Denote the relative completion of the mapping class group of $(\Sbar,P,\vv_o)$ by $\cG_{g,n+\uu}$ and its prounipotent radical by $\U_{g,n+\uu}$. (See \cite{hain:torelli} for definitions.) These act on $\Q\pi_1(S,\vv_o)^\wedge$. Denote the image of $\U_{g,n+\uu}$ in $\Aut\Q\pi_1(S,\vv_o)^\wedge$ by $\Ubar_{g,n+\uu}$.\footnote{Conjecturally, the homomorphism $\cG_{g,n+\uu} \to \Aut\Q\pi_1(S,\vv_o)^\wedge$ is injective, which would imply that $\U_{g,n+\uu} = \Ubar_{g,n+\uu}$.} The vector field $\xi_o$ determines a homomorphism $\Ubar_{g,n+\uu} \to H_1(\Sbar)$ that depends only on the vector $\dd$ of local degrees of $\xi$. Denote its kernel by $\Ubar_{g,n+\uu}^\dd$.\footnote{Explicit presentations of the Lie algebras of the $\U_{g,n+\uu}$ are known for all $n\ge 0$ when $g \neq 2$ \cite{hain:torelli,hain:kahler,hain-matsumoto:mem}; partial presentations (e.g., generating sets) are known when $g=2$, \cite{watanabe}. Presentations of the $\U_{g,n+\uu}^\dd$ can be easily deduced from these.} The group $\Uhat_{g,n+\uu}^\dd$, mentioned above, is the subgroup of $\KRV_{g,n+\uu}^\dd$ generated by $\Ubar_{g,n+\uu}^\dd$ and $\U^\MT_{X,\vv}$.

Ihara and Nakamura \cite{ihara-nakamura} construct canonical smoothings of each maximally degenerate stable curve $X_0$ of type $(g,n+1)$ over $\Z[[q_1,\dots,q_N]]$ for all $n \ge 0$, where $N= \dim \M_{g,n+1}$. Associated to each tangent vector $\vv = \pm \partial/\partial q_j$ of $\Mbar_{g,n+1}$ at the point corresponding to $X_0$, there is a limit pro-MHS on $\Q\lambda(X)^\wedge$, that we denote by $\Q\lambda(X_\vv)^\wedge$.

\begin{hypothesis}
\label{hypoth:ihara}
The limit MHS on $\Q\lambda(X_\vv)^\wedge$ is the Hodge realization of a pro-object of $\MTM(\Z)$. Equivalently, the Mumford--Tate group of the MHS on $\Q\lambda(X_\vv)^\wedge$ is isomorphic to $\pi_1(\MTM,\w^B)$.
\end{hypothesis}

A more detailed description of this hypothesis is given in the proof of Proposition~\ref{prop:surjection}. The author claims that this hypothesis is true. The proof is expected to be available in \cite{hain:ihara}. Brown's result \cite{brown} is a significant ingredient in the proof, and also the $(0,3)$ case.

\begin{bigtheorem}
\label{thm:krv}
If $2g+n>1$ (i.e., $S$ is hyperbolic), then the group $\Uhat_{g,n+\uu}^\dd$ does not depend on the choice of a quasi-algebraic structure $\phi : (\Sbar,P,\vv_o,\xi_o) \to (\Xbar,D,\vv,\xi)$. The group $\Ubar_{g,n+\uu}^\dd$ is normal in $\Uhat_{g,n+\uu}^\dd$. If we assume Hypothesis~\ref{hypoth:ihara}, there is a canonical surjective group homomorphism $\K \to \Uhat_{g,n+\uu}^\dd/\Ubar_{g,n+\uu}^\dd$, where $\K$ denotes the prounipotent radical of $\pi_1(\MTM)$.
\end{bigtheorem}

This result follows from a more general result, Theorem~\ref{thm:krv2}, which is proved in Section~\ref{sec:krv}. We expect the homomorphism $\K \to \Uhat_{g,n+\uu}^\dd/\Ubar_{g,n+\uu}^\dd$ to be an isomorphism. The injectivity of this homomorphism is closely related to Oda's Conjecture \cite{oda} (proved in \cite{takao}) and should follow from it.

In genus 1, the associated graded Lie algebra $\Gr^W_\bdot \u_{1,n+\uu}$ of the Lie algebra of $\U_{1,n+\uu}$ contains the derivations $\delta_{2n}$ (denoted $\e_{2n}$ in \cite{hain-matsumoto:mem}). The first statement of Theorem~\ref{thm:krv} implies \cite[Thm.~1.5]{akkn2} as well as higher genus generalizations.

\begin{conjecture}
The inclusion $\Uhat_{g,n+\uu}^\dd \to \KRV_{g,n+\uu}^\dd$ is an isomorphism if and only if the inclusion of $\pi_1(\MTM)$ into $\GRT$, the de~Rham version of Grothendieck--Teichm\"uller group, is an isomorphism. In this case, $\KRV_{g,n+\uu}^\dd$ will be a split extension
$$
1 \to \Ubar_{g,n+\uu}^\dd \to \KRV_{g,n+\uu}^\dd \to \K \to 1.
$$
\end{conjecture}

\begin{remark}
The precise relationship between Theorem~\ref{thm:krv} and the conjecture is not clear to the author. The group $\GRT$ is an extension of $\Gm$ by a prounipotent group that we denote by $\URT$. Brown's Theorem \cite{brown} implies that there is an inclusion $\K \hookrightarrow \URT$. It is shown in \cite{akkn2} that there is a natural inclusion $\URT \hookrightarrow \KRV^\dd_{g,n+\uu}$ when $2g+n-1>0$. To understand the relation between the theorem and the conjecture, one has to understand the ``geometric part'' $\URT\cap \Ubar_{g,n+\uu}^\dd$ of $\GRT$. If it is trivial, then $\Uhat_{g,n+uu}^\dd \to \KRV_{g,n+\uu}^\dd$ would be an isomorphism if and only if $\K \to \URT$ were an isomorphism (or, equivalently, $\pi_1(\MTM) \to \GRT$ were an isomorphism).
\end{remark}

A few remarks about the approach and the structure of the paper. As when proving that the Goldman bracket is a morphism of MHS \cite{hain:goldman}, the proof of Theorem~\ref{thm:turaev} consists in:
\begin{enumerate}

\item Finding a homological description of the cobracket $\delta_\xi$ analogous to the homological description of the Goldman bracket given by Kawazumi--Kuno \cite[\S3]{kk:log}. This description gives a factorization of the cobracket.

\item Giving a de~Rham description of the continuous dual of each map in this factorization.

\item Proving that, for each quasi-complex structure on $(\Sbar,P,\vv_o,\xi_o)$, each map in this factorization of the dual cobracket is a morphism of MHS.

\end{enumerate}
The homological description of the cobracket is established in Sections~\ref{sec:factorizn} and \ref{sec:turaev}. This description appears to be new. The de~Rham descriptions of the factors of the dual cobracket are given in Section~\ref{sec:DR}. The proof of Theorem~\ref{thm:turaev} is completed in Section~\ref{sec:hodge_turaev} where it is shown that each map in the factorization of the cobracket is a morphism of MHS for each choice of a complex structure. The group $\Uhat_{g,n+\uu}^\dd$ is defined and analyzed in Section~\ref{sec:torsors}, and Theorem~\ref{thm:krv} is proved in Section~\ref{sec:krv}.

This paper is a continuation of \cite{hain:goldman}. We assume familiarity with the sections of that paper on rational $K(\pi,1)$ spaces, iterated integrals, and Hodge theory.

\bigskip

\noindent{\em Acknowledgments:} I would like to thank Anton Alekseev, Nariya Kawazumi, Yusuke Kuno and Florian Naef for patiently answering my numerous questions about their work. I am also grateful to Quentin Gendron for pointing out an issue with the existence of algebraic framings which necessitated the introduction of quasi-algebraic framings. Finally, I am very grateful to the referees for their comments, suggestions and for pointing out numerous typographical errors and points that needed clarification.

\section{Notation and Conventions}

Suppose that $X$ is a topological space. There are two conventions for multiplying paths. We use the topologist's convention: The product $\alpha\beta$ of two paths $\alpha,\beta : [0,1]\to X$ is defined when $\alpha(1) = \beta(0)$. The product path traverses $\alpha$ first, then $\beta$. We will denote the set of homotopy classes of paths from $x$ to $y$ in $X$ by $\pi(X;x,y)$. In particular, $\pi_1(X,x) = \pi(X;x,x)$. The fundamental groupoid of $X$ is the category whose objects are $x\in X$ and where $\Hom(x,y) = \pi(X;x,y)$.

As in \cite{hain:goldman}, we have attempted to denote complex algebraic and analytic varieties by the roman letters $X$, $Y$, etc and arbitrary smooth manifolds (and differentiable spaces) by the letters $M$, $N$, etc. This is not always possible. The diagonal in $T\times T$ will be denoted $\Delta_T$.

The singular homology of a smooth manifold $M$ will be computed using the complex $C_\bdot(M)$ of {\em smooth} singular chains. The complex $C^\bdot(M)$ will denote its dual, the complex of smooth singular cochains. The de~Rham complex of $M$ will be denoted by $E^\bdot(M)$. The integration map $E^\bdot(M) \to C^\bdot(M;\R)$ is thus a well-defined cochain map.

The notation $\bil : C^\bdot \otimes C_\bdot \to \kk$ will often be used for any (natural) pairing between a chain complex $C_\bdot$ and a dual cochain complex $C^\bdot$. For example, it will be used to denote Kronecker pairings and integration pairings.

\subsection{Local systems and connections}

Here we regard a local system on a manifold $N$ as a locally constant sheaf. We will denote the complex of differential forms on $N$ with values in a local system $\bV$ of real (or rational) vector spaces by $E^\bdot(N;\bV)$. As in \cite{hain:goldman}, we denote the flat vector bundle associated to $\bV$ by $\sV$ and the sheaf of $j$-forms on $N$ with values in $\bV$ by $\sE^j_N\otimes\sV$. So $E^j(N,\bV)$ is just the space of global sections of $\sE^j_N\otimes\sV$. There are therefore isomorphisms
$$
H^\bdot(E^\bdot(N;\bV)) \cong H^\bdot(N;\bV)
$$

The pullback of a local system $\bV$ over $Y\times Y$ along the interchange map $\tau : Y^2 \to Y^2$ will be denoted by $\bV^\op$.

\subsection{Cones}
\label{sec:cones}

Several homological constructions will use cones. Since signs are important, we fix our conventions. The cone of a map $\phi : A_\bdot \to B_\bdot$ of chain complexes is defined by
$$
C_\bdot(\phi) := \cone(A_\bdot \to B_\bdot)[-1],
$$
where $C_j(\phi) = B_j \oplus A_{j-1}$ with differential $\partial (b,a) = (\partial b - \phi(a), - \partial a)$. The cone of a map $\psi: \cB^\bdot \to \cA^\bdot$ of cochain complexes is defined by
$$
C^\bdot(\psi) := \cone(\cB^\bdot\to \cA^\bdot)[-1],
$$
where $C^j(\psi) := \cB^j \oplus \cA^{j-1}$ with differential $d(\beta,\alpha) = (d\beta,-d\alpha - \psi^\ast\beta)$. Pairings of complexes
$$
\bil_A : \cA^\bdot\otimes A_\bdot \to V \text{ and } \bil_B : \cB^\bdot \otimes B_\bdot \to V
$$
induce the pairing
$$
\bil : C^\bdot(\psi) \otimes C_\bdot(\phi) \to V
$$
defined by $(\beta,\alpha)\otimes (b,a) \mapsto \langle \alpha, a\rangle_A + \langle\beta,b\rangle_B$. It satisfies $\langle dz,c\rangle = \langle z,\partial c\rangle$ and thus induces a pairing
$$
\bil : H^\bdot(C^\bdot(\psi)) \otimes H_\bdot(C_\bdot(\phi)) \to V.
$$

\section{Preliminaries}

We recall and elaborate on notation from \cite{hain:goldman}. Fix a ring $\kk$. Typically, this will be $\Z$, $\Q$, $\R$ or $\C$. Suppose that $M$ is a smooth manifold, possibly with boundary. All paths $[0,1]\to M$ will be piecewise smooth unless otherwise noted. Denote the space of paths $\gamma : [0,1]\to M$ by $PM$. This is endowed with the compact open topology. For each $t\in [0,1]$, one has the map
$$
p_t : PM \to M
$$
defined by $p_t(\gamma) = \gamma(t)$. It is a (Hurewicz) fibration.

\subsection{Fibrations}
 
The most fundamental path fibration is the map
\begin{equation}
\label{eqn:path_fibn}
p_0\times p_1 : PM \to M\times M.
\end{equation}
Its fiber over $(x_0,x_1)$ is the space $P_{x_0,x_1}M$ of paths in $M$ from $x_0$ to $x_1$. When $x_0=x_1 = x$, the fiber is the space $\Lambda_x M$ of loops in $M$ based at $x$. The local system over $M\times M$ whose fiber over $(x_0,x_1)$ is $H_0(P_{x_0,x_1}M;\kk)$ will be denoted by $\bP_M$.

 More generally, for $(t_1,\dots,t_n) \in [0,1]^n$ with $0 < t_1 \le t_2 \le \dots \le t_n < 1$, one has the fibration
$$
\prod_{j=1}^n p_{t_j} : PM \to M^n
$$ 
whose fiber over $(x_1,\dots,x_n)$ is
$$
P_{\blank,x_1}M \times P_{x_1,x_2}M \times \dots \times P_{x_{n-1},x_n}M \times P_{x_n,\blank} M.
$$
Here $P_{\blank,x}M$ denotes the space of paths terminating at $x \in M$ and $P_{x,\blank}M$ denotes the space of paths emanating from $x$. Since $P_{x,\blank}M$ and $P_{\blank,x}M$ are contractible, the fiber of the corresponding local system over $M^n$ is
$$
\pi_{1,2}^\ast\bP_M\otimes \pi_{2,3}^\ast\bP_M \otimes \cdots \otimes \pi_{n-1,n}\bP_M,
$$
where $\pi_{j,k} : M^n \to M\times M$ denotes the projection onto the product of the $j$th and $k$th factors.

The ``pullback path fibration'' obtained by pulling back (\ref{eqn:path_fibn}) along a smooth map $f : N \to M\times M$ will be denoted by $P_f M \to N$. When $f$ is the diagonal map $M \to M\times M$, the pullback is the fibration
\begin{equation}
\label{eqn:Lambda_over_M}
p : \Lambda M \to M
\end{equation}
of the free loop space of $M$ over $M$. Its fiber over $x\in M$ is the space $\Lambda_x M$ of loops in $M$ based at $x$. The corresponding local system will be denoted by $\bL_M$. It has fiber $H_0(\Lambda_x M;\kk)$ over $x\in M$.

\subsection{Homology}

The following result follows easily from the fact that a non-compact surface is a $K(\pi,1)$ and has cohomological dimension 1. Cf.\ \cite[Prop.~3.5]{hain:goldman}.

\begin{proposition}
\label{prop:vanishing}
If $M$ is a surface and if $M$ is not closed, then $H_j(\Lambda M)$ vanishes (with all coefficients) for all $j>1$. \qed
\end{proposition}

\section{Factoring Loops}
\label{sec:factorizn}

In this section $M$ is a smooth manifold and $\kk$ is any commutative ring. Recall from \cite[\S3.3]{hain:goldman} the construction of the Chas--Sullivan map
$$
\cs : H_0(\Lambda M) \to H_1(M;\bL_M).
$$
It is induced by the map that takes a loop $\alpha : S^1 \to M$ to the horizontal lift $\alphahat : S^1 \to \bL_M$ of $\alpha$ defined by $\alphahat(\theta)(\phi) = \alpha(\phi+\theta)$. 

We now describe a generalization of the Chas--Sullivan map. It arises from the factorization of a loop into two arcs. The evaluation map
\begin{equation}
\label{eqn:fibn}
p_0 \times p_{1/2} : \Lambda M \to M\times M
\end{equation}
is a fibration. Its fiber over $(x,y)$ is $P_{x,y}M \times P_{y,x} M$. The corresponding local system over $M\times M$ is
$$
\bP_M \otimes \bP_M^\op
$$
where $\bV^\op$ denotes the pullback of the local system $\bV$ on $M\times M$ along the map $(x,y) \mapsto (y,x)$. The restriction of $\bP_M\otimes \bP_M^\op$ to the diagonal $\Delta_M \cong M$, is $\bL_M\otimes \bL_M$.

\begin{remark}
\label{rem:K(pi,1)}
When $M$ is a $K(\pi,1)$, each path component of $P_{x_0,x_1}M$ is contractible. The Leray-Serre spectral sequences of the fibrations (\ref{eqn:Lambda_over_M}) and (\ref{eqn:fibn}) imply that there are natural isomorphisms
$$
H_\bdot(M;\bL_M) \cong H_\bdot(\Lambda M) \cong H_\bdot(M^2;\bP_M \otimes \bP_M^\op).
$$
\end{remark}

Composing $\cs$ with the maps induced on homology by the two maps $\bL_M \to \bL_M \otimes \bL_M$ defined by
$$
\alpha \mapsto \alphahat \otimes 1_{p(\alpha)} \text{ and } \alphahat \mapsto 1_{p(\alpha)} \otimes \alpha,
$$
where $1$ denotes the horizontal section of $\bL_M$ whose value at $x$ is $1_x$, gives two maps
$$
\cs \otimes 1 \text{ and } 1 \otimes \cs:  H_0(\Lambda M) \To H_1(M;\bL_M\otimes \bL_M).
$$
Composing these with the diagonal map
$$
\Delta_\ast : H_1(M;\bL_M\otimes \bL_M) \To H_1(M^2; \bP_M\otimes\bP_M^\op)
$$
yields two maps
$$
\Delta_\ast(\cs \otimes 1) \text{ and } \Delta_\ast(1 \otimes \cs) :
H_0(\Lambda M) \To H_1(M^2; \bP_M\otimes\bP_M^\op).
$$

\begin{proposition}
\label{prop:homologous}
There is a horizontal section $s_\alpha : [0,2\pi i]\times S^1$ of $\bP_M\otimes \bP_M^\op$ satisfying
$$
\partial s_\alpha = 1\otimes\alphahat - \alphahat \otimes 1.
$$
Consequently, $\Delta_\ast(\cs \otimes 1) = \Delta_\ast(1 \otimes \cs)$.
\end{proposition}

\begin{proof}
Each loop $\alpha : S^1 \to M$ induces a map $\alpha^2 : S^1 \times S^1 \to M \times M$. This lifts to the horizontal section $s_\alpha$ of $\bP_M\otimes\bP_M^\op$ defined over $(S^1 \times S^1) - \Delta_{S^1}$ defined by
$$
s_\alpha(\theta,\phi) = \alpha'\otimes \alpha'',
$$
where $\alpha'$ is the restriction of $\alpha$ to the positively oriented arc in $S^1$ from $\theta$ to $\phi$ and $\alpha''$ is its restriction to the arc from $\phi$ to $\theta$. This lift does not extend continuously to $S^1\times S^1$, except when $\alpha$ is null homotopic.

To extend the lift, we replace $S^1\times S^1$ by $U := [0,2\pi] \times S^1$. The map
\begin{equation}
\label{eqn:quot}
U \to S^1 \times S^1,\qquad (t,\phi) \mapsto (t+\phi, \phi)
\end{equation}
is a quotient map that takes the boundary of $U$ onto the diagonal $\Delta_{S^1}$. It induces a homeomorphism $(0,2\pi) \times S^1 \approx (S^1\times S^1) - \Delta$ and identifies $(0,\phi)$ with $(2\pi,\phi)$. The horizontal lift $s_\alpha : (S^1\times S^1) - \Delta_{S^1} \to \bP_M\otimes \bP_M^\op$ of $\alpha^2$ extends uniquely to a horizontal lift
$$
U \to \bP_M\otimes \bP_M^\op
$$
which we will also denote by $s_\alpha$. The boundary of $U$ is $\{2\pi\}\times S^1 -  \{0\}\times S^1$, which implies that $\partial s_\alpha = 1\otimes\alphahat - \alphahat \otimes 1$.
\end{proof}

When $M$ is a $K(\pi,1)$, both maps $\Delta_\ast(\cs \otimes 1)$ and $\Delta_\ast(1 \otimes \cs)$ are easily seen to correspond to $\cs : H_0(\Lambda M) \to H_1(M,\bL_M) \cong H_1(\Lambda M)$ under the isomorphism given in Remark~\ref{rem:K(pi,1)}.

\section{A Homological Description of the Turaev Cobracket}
\label{sec:turaev}

Throughout this section, $M$ will be a smooth oriented surface without boundary\footnote{If $M$ has boundary, replace it by $M-\partial M$, which is homotopy equivalent to $M$.} and $\kk$ arbitrary. Denote space of non-zero tangent vectors of $M$ by $\Mhat$ and the projection by $\pi : \Mhat \to M$. Denote the composition of the projection $\pi : \Mhat \to M$ with the diagonal map $\Delta : M\to M\times M$ by $\Deltabar$.

\begin{remark}
When reading this section, it is worth keeping in mind that, by an elementary case of a theorem of Hirsch \cite{hirsch} (that goes back to Whitney \cite{whitney}), the set of regular homotopy classes of immersed loops in $M$ corresponds to the set $\pi_0(\Lambda \Mhat)$ of homotopy classes of loops in $\Mhat$. The formula (\ref{eqn:def_turaev}) gives a well-defined map
$$
H_0(\Lambda \Mhat) \to H_0(\Lambda M) \otimes H_0(\Lambda M).
$$
The key point in this section is to give a homological description of this map.
\end{remark}

\subsection{A homological description of $\cs : H_0(\Lambda \Mhat) \to H_1(\Lambda\Mhat)$}
The first step in giving a homological description of the cobracket is to give a homological description of the homology of $\Lambda \Mhat$ that is suitable for computing the intersection product. It uses a cone construction and the homotopy $s_\alpha$ constructed in Proposition~\ref{prop:homologous}.

Define
\begin{equation}
\label{eqn:iota}
\iota : \bL_\Mhat \to \Deltabar^\ast \big(\bP_M \otimes \bP_M^\op\big)
\end{equation}
to be the map whose restriction to the fiber $\bL_{\Mhat,v}$ over $v \in \Mhat$ is defined by
$$
\iota(\alpha) = 1_x \otimes (\pi\circ \alpha) - (\pi\circ \alpha) \otimes 1_x
\in H_0(\Lambda_x M) \otimes H_0(\Lambda_x M),
$$
where $\alpha \in \Lambda_v \Mhat$ and $x=\pi(v)$.

The maps $\Deltabar$ and $\iota$ induce a chain map
$$
\Deltabar_\ast\otimes \iota : C_\bdot(\Mhat;\bL_\Mhat) \to C_\bdot(M^2;\bP_M\otimes \bP_M^\op)
$$
of singular chain complexes. We can therefore form the cone
$$
C_\bdot(\Deltabar_\ast\otimes \iota) := \cone\big(C_\bdot(\Mhat;\bL_\Mhat) \to C_\bdot(M^2;\bP_M\otimes \bP_M^\op)\big)[-1]
$$
Set
$$
H_\bdot(M^2,\Mhat;\bL_\Mhat\to \bP_M\otimes\bP_M^\op) := H_\bdot(C_\bdot(\Deltabar_\ast\otimes \iota))
$$

For each $\alpha \in \Lambda \Mhat$ we have the 2-cycle $(s_{\pi\circ\alpha},\alphahat) \in C_2(\Deltabar_\ast\otimes \iota)$, where
$$
s_{\pi\circ\alpha} : U \to \bP_M\otimes\bP_M^\op
$$
is the section defined in Proposition~\ref{prop:homologous}.

\begin{proposition}
\label{prop:phi_cs}
The map that takes the class of a loop $\alpha \in \Lambda \Mhat$ to the class of the cycle $(s_{\pi\circ\alpha},\alphahat) \in C_\bdot(\Deltabar_\ast\otimes \iota)$ defines a homomorphism
\begin{equation}
\label{eqn:phi}
\varphi : H_0(\Lambda \Mhat) \to H_2(M^2,\Mhat;\bL_\Mhat\to \bP_M\otimes\bP_M^\op)
\end{equation}
whose composition with the map $H_2\big(M^2,\Mhat;\bL_\Mhat\to \bP_M\otimes\bP_M^\op\big) \to H_1(\Mhat;\bL_\Mhat)$ is the Chas--Sullivan map $\cs$ for $\Mhat$.
\end{proposition}

The following result is he homological description of $\cs$ that we will need in the homological description of the Turaev cobracket.

\begin{lemma}
\label{lem:les_cone}
If $M$ is a surface that is not $S^2$, then there is a diagram
$$
\xymatrix{
&& H_0(\Lambda \Mhat)\ar[d]_\varphi \ar[dr]^\cs \cr
\ar[r]_(.3)\partial & H_2(\Lambda M) \ar[r] & H_2(M^2,\Mhat;\bL_\Mhat\to \bP_M\otimes\bP_M^\op) \ar[r]_(.7)\psi & H_1(\Lambda \Mhat) \ar[r]_(.6)\partial & 0
}
$$
whose bottom row is exact. If $M$ is not closed, then $H_2(\Lambda M)$ vanishes so that the map $\psi$ is an isomorphism.
\end{lemma}

\begin{proof}
The bottom row is part of the long exact homology sequence associated to the cone $C_\bdot(\Deltabar_\ast\otimes \iota)$ after making the identifications from Remark~\ref{rem:K(pi,1)}. Under these identifications, the connecting homomorphism
$$
\partial : H_\bdot(\Lambda \Mhat) \to H_\bdot(\Lambda M)
$$
vanishes as it is induced by the map $\alpha \mapsto \pi\circ\alpha - \pi\circ\alpha$.
\end{proof}

\subsection{The groups $H^\bdot_\Delta(M^2,N)$}
\label{sec:rel_coho}

Denote the singular cochain complex of a pair $(Y,Z)$ with coefficients in $\kk$ by $C^\bdot(Y,Z)$. For a continuous map $h:T\to M^2$, define
$$
C^\bdot_\Delta(M^2,T) := \cone\big(C^\bdot(M^2,M^2-\Delta_M) \overset{h^\ast}{\To} C^\bdot(T)\big)[-1].
$$
Denote its cohomology groups by $H^\bdot_\Delta(M^2,T)$. These can also be computed by the complex
$$
\cone\big(C^\bdot(M^2) \overset{j^\ast \oplus h^\ast}\To C^\bdot(M^2-\Delta_M) \oplus C^\bdot(T)\big)[-1],
$$
where $j : M^2-\Delta_M \to M^2$ is the inclusion. 

\begin{lemma}
\label{lem:les}
There is a long exact sequence
$$
\xymatrix{
\cdots \ar[r] & H^{j-1}(T) \ar[r] & H^j_\Delta(M^2,T) \ar[r] & H^j_\Delta(M^2) \ar[r] & H^j(T) \ar[r] & \cdots
}.
$$
\end{lemma}

\begin{proof}
The long exact sequence comes from the short exact sequence
$$
\xymatrix{
0 \ar[r] & C^\bdot(T)[-1] \ar[r] & C^\bdot_\Delta(M^2,T) \ar[r] & C^\bdot_\Delta(M^2) \ar[r] & 0
}
$$
of complexes.
\end{proof}

We are interested in the 3 cases: $T$ is empty; $T=\Delta_M$ and $h$ is the inclusion; $T=\Mhat$ and $h$ is the composition of the projection $\pi$ with the diagonal map. When $T$ is empty, the Thom isomorphism gives an isomorphism $H^j(M) \cong H^{j+2}_\Delta(M^2)$. We'll consider the case $T=\Mhat$ in the next section. Here we consider the case $T=\Delta_M$.

We will suppose that $\xi$ is a nowhere vanishing vector field on $M$. The normal bundle of the diagonal $\Delta_M$ in $M^2$ is isomorphic to the tangent bundle $TM$ of $M$. Fix a riemannian metric on $M$. The exponential map induces a diffeomorphism of a closed disk bundle in $TM$ with a regular neighbourhood $N$ of $\Delta_M$ in $M^2$. By rescalling $\xi$, we may assume that the exponential map takes $\xi$ into $\partial N$. We will henceforth regard $\xi$ as the section $\exp \xi$ of $\partial N$. Denote the closed unit ball in $\R^2$ by $B$. We can choose a trivialization
$$
\pi \times q : N \overset{\simeq}{\To} M \times B
$$
such that $q\circ \xi : M \to B-\{0\}$ is null homotopic. This condition determines the homotopy class of the trivialization.

Since $\partial M$ is empty, the inclusion $(N,\partial N) \hookrightarrow (M^2,M^2-\Delta_M)$ induces an isomorphism
$$
H^\bdot_\Delta(M^2,\Delta_M) \overset{\simeq}{\To} H^\bdot(N,\Delta_M \cup \partial N).
$$
The K\"unneth Theorem implies that
$
q^\ast : H^2(B,\partial B) \to H^2(N,\partial N)
$
is an isomorphism.

\begin{proposition}
There is a short exact sequence
$$
0 \to H^1(\Delta_M) \to H^2(N,\Delta_M \cup \partial N) \to H^2(N,\partial N) \to 0.
$$
\end{proposition}

\begin{proof}
This is part of the long exact sequence of the triple $(N,\Delta_M \cup \partial N,\partial N)$. Exactness on the left follows from the K\"unneth Theorem (or the Thom Isomorphism Theorem); exactness on the right follows as $\Delta_M \hookrightarrow N \overset{q}{\To} B$ is the constant map $0$.
\end{proof}

The projection $q : N \to B$ induces a homomorphism
$$
q^\ast : H^2(B,\partial B) \cong H^2(B,\{0\}\cup \partial B) \to H^2(N,\Delta_M \cup \partial N).
$$
This map depends on the homotopy class of the trivialization $\xi$. Denote the positive integral generator of $H^2(B,\partial B)$ by $\tau_B$. Define
$$
\tau_\xi := q^\ast \tau_B \in H^2(N,\Delta_M\cup \partial N).
$$
The image of $\tau_\xi$ in $H^2(N,\partial N)$ is the Thom class $\tau_M$ of the tangent bundle of $M$.

Choose a representative $\w_B \in C^2(B,\partial B)$ of $\tau_B$. Set $\w_\xi = q^\ast \w_B$. Since the restriction of $\w_B$ to the diagonal vanishes, the pair $(\w_\xi,0)$ represents $\tau_\xi \in H^2(N,\Delta_M \cup \partial N)$.

To better understand $\tau_\xi$, suppose that $\gamma : S^1 \to \partial N$. Since $\partial M$ is empty, there is a diffeomorphism $\partial N \approx M \times \partial B$ that commutes with the projections to $M$. Define the rotation number $\rot_\xi(\gamma)$ of $\gamma$ with respect to $\xi$ to be the rotation number of $q\circ \gamma$ about $0 \in B$. If $\gammadot: S^1 \to \Mhat$ is the tangent map lift of  an immersed loop $\gamma$ in $M$, then $\rot_\xi(\gammadot)$ equals the standard rotation number $\rot_\xi(\gamma)$.

Let $\Gamma_\gamma$ be the relative 2-cycle
$$
\Gamma_\gamma : (I\times S^1,\partial I \times S^1) \to (N,\Delta_M\cup \partial N) \approx M\times (B,\{0\} \cup \partial B)
$$
that corresponds to the map
$$
(I\times S^1,\partial I \times S^1) \to M\times (B,\{0\}\cup\partial B), \quad
(t,\theta) \mapsto \big(\pi\circ\gamma(t),t\, q\circ\gamma(\theta)\big).
$$
Give $I\times S^1$ the product orientation. 

\begin{lemma}
\label{lem:rot}
We have $\langle \tau_\xi,\Gamma_\gamma \rangle = \rot_\xi(\gamma)$.
\end{lemma}

\begin{proof}
Write $\w_B = d\eta_B$ in $C^2(B)$. Observe that $\rot_\xi(\gamma) = \langle \eta_B,q\circ\gamma\rangle$. Since $\partial \Gamma_\gamma = \gamma - c_0$, where $c_0$ denotes the constant map $S^1 \to B$ with value 0,
$$
\langle \tau_\xi, \Gamma_\gamma \rangle
= \langle \w_B, q_\ast\Gamma_\gamma \rangle
= \langle d\eta_B, q_\ast \Gamma_\gamma \rangle
= \langle \eta_B, q_\ast \partial \Gamma_\gamma \rangle
= \langle \eta_B, q\circ\gamma\rangle
= \rot_\xi(\gamma).
$$
\end{proof}

\subsection{The class $c_\xi$}
\label{sec:c}

In this section, we show that each non-vanishing vector field $\xi$ determines a class $c_\xi \in H^2_\Delta(M^2,\Mhat)$. Pairing with this class corresponds to intersecting with the diagonal and is a key component of the homological description of $\delta_\xi$.

\begin{lemma}
\label{lem:f_xi}
Each section $\xi$ of $\Mhat \to M$ determines a class $f_\xi \in H^1(\Mhat;\Z)$ whose pullback $\xi^\ast f_\xi$ to $M$ vanishes and whose restriction to each fiber $\Mhat_x$ is the positive integral generator of $H^1(\Mhat_x;\Z)$. It is characterized by these properties. If $\gammadot: S^1 \to \Mhat$ is the tangent map lift of  an immersed loop $\gamma$ in $M$, then $\rot_\xi(\gamma) = \langle f_\xi, \gammadot \rangle$.
\end{lemma}

\begin{proof}
This follows from the K\"unneth Theorem and the fact the section $\xi$ determines a trivialization $r:\Mhat \overset{\simeq}{\To} M \times (\R^2-\{0\})$ with $r\circ \xi$ constant. It is unique up to homotopy. Take $f_\xi$ to be the pullback of the positive generator of $H^1(S^1;\Z)$ under the projection $\Mhat \overset{\simeq}{\To} M \times (\R^2-\{0\}) \to \R^2-\{0\} \to S^1$. The last statement follows from Lemma~\ref{lem:rot}.
\end{proof}

\begin{lemma}
\label{lem:s}
When $\Mhat \to M$ is a trivial bundle, there is a short exact sequence
$$
\xymatrix{
0 \ar[r] & H^1(\Mhat) \ar[r] & H^2_\Delta(M^2,\Mhat) \ar[r] & H^2_\Delta(M^2) \ar[r]  & 0
}.
$$
Each framing $\xi$ of $M$ induces a natural splitting $s_\xi : H^2_\Delta(M^2) \to H^2_\Delta(M^2,\Mhat)$ which depends only on the homotopy class of $\xi$.
\end{lemma}

\begin{proof}
This is part of the long exact sequence in Lemma~\ref{lem:les}. Exactness of the sequence follows from the Thom isomorphism $H^j(\Delta_M) \cong H_\Delta^{j+2}(M^2)$, which implies that $H^1_\Delta(M^2) = 0$. The triviality of the tangent bundle of $M$ implies that the normal bundle of the diagonal in $M^2$ is trivial, which gives the exactness on the right.

Since $H^2_\Delta(M^2)$ is freely generated by the Thom class $\tau_M$ of $M$, to construct the lift, it suffices to lift $\tau_M$ to $H^2_\Delta(M^2,\Mhat)$. To do this, note that $\pi : \Mhat \to \Delta_M$ induces a map
$$
\pi^\ast : H^2_\Delta(M^2,\Delta_M) \to H^2_\Delta(M^2,\Mhat)
$$
and recall that $H^2_\Delta(M^2,\Delta_M) \cong H^2(N,\Delta_M\cup\partial N)$. Define $s_\xi(\tau_M) = \pi^\ast \tau_\xi$.
\end{proof}

\begin{definition}
Define  
$
c_\xi := \pi^\ast \tau_\xi + f_\xi \in H^2_\Delta(M^2,\Mhat),
$
where $f_\xi\in H^1(\Mhat)$ is identified with its image in $H^2_\Delta(M^2,\Mhat)$.
\end{definition}

\subsection{The pairing}
\label{sec:pairing}

Here we define a pairing and compute its value on $\varphi(\alpha)\otimes c_\xi$ for all $\alpha \in \Lambda \Mhat$. It is closely related to the value $\delta_\xi(\pi\circ\alpha)$ of the cobracket on $\alpha$.

\begin{proposition}
\label{prop:pairing}
There is a well-defined pairing
$$
\blank\cap\blank :
H_2(M^2,\Mhat;\bL_\Mhat\to \bP_M\otimes\bP_M^\op)\otimes H^2_\Delta(M^2,\Mhat) \to H_0(M;\bL_M\otimes \bL_M).
$$
\end{proposition}

\begin{proof}
We continue with the notation above. Let $U=N-\partial N$. Let $r : U \to \Delta_M$ be a retraction. Let $\U = \{M^2 - \Delta_M,U\}$. It is an open cover of $M^2$. We can compute the product using $\U$-small chains $C_\bdot^\U$ and cochains $C^\bdot_\U$ via the pairing
{\scriptsize\begin{multline*}
\cone\big(C_\bdot(\Mhat;\bL_\Mhat) \to C_\bdot^\U(M^2;\bP_M\otimes \bP_M^\op)\big)[-1] \otimes \cone\big(C^\bdot_\U(M^2,M^2-\Delta_M)
\to C^\bdot(\Mhat)\big)[-1]
\cr
\to C_\bdot(U;\bP_M\otimes \bP_M^\op)
\end{multline*}
}
defined by
$$
(s,u)\cap (\zeta,\eta) = \langle \zeta,s \rangle + \iota\langle \eta,u\rangle,
$$
where, as usual, $\bil$ denotes the natural pairing between cochains and chains. The induced pairing between $H_2$ and $H^2$ takes values in $H_0(U,(\bP_M\otimes \bP_M^\op)|_U)$. This group is naturally isomorphic to $H_0(M;\bL_M\otimes\bL_M)$ as the homotopy equivalence $r:U\to M$ induces a natural isomorphism $r^\ast (\bL_M\otimes \bL_M) \cong (\bP_M\otimes \bP_M^\op)|_U$.
\end{proof}

The following computation is the main ingredient in the computation of $\varphi(\alpha)\cap c_\xi$, where it is used with $\alpha \in \Lambda \Mhat$ and $\beta = \pi\circ \alpha$. Recall from the introduction the notation for $\e_P$, $\beta_P'$ and $\beta_P''$. Recall that $\w_\xi = q^\ast \w_B$ is a 2-cocycle that represents $c_\xi$.

\begin{lemma}
\label{lem:formula}
If $\beta : S^1 \to M$ is an immersed circle with transverse self intersections, then
$$
\langle \w_\xi, s_\beta \rangle = 
\rot_\xi(\beta)\big(\beta\otimes 1 - 1 \otimes \beta \big)
-\sum_P \e_P(\beta_P'\otimes \beta_P'' - \beta_P''\otimes \beta_P') \in H_0(\Lambda M)^{\otimes 2},
$$
where $P$ ranges over the double points of $\beta$.
\end{lemma}

\begin{proof}

We use the notation of Section~\ref{sec:rel_coho}. Since the map $\beta^2 : S^1 \times S^1 \to M^2$ maps the diagonal $\Delta_{S^1}$ in $S^1\times S^1$ to the diagonal in $M^2$, $\beta^2$ cannot be transverse to $\Delta_M$. However, by shrinking $N$ if necessary, we may assume that $\beta$ is transverse to $\partial N_r$ for all $0<r\le 1$, where $N_r$ denotes disk sub-bundle of $N$ of radius $r$, where $N = N_1$. In this case, the inverse image of $N$ under $\beta^2$ is a disjoint union
$$
\Gamma \cupdot \coprod_{(\theta,\phi)} U_{\theta,\phi}
$$
where $\Gamma$ is a neighbourhood of $\Delta_{S^1}$ diffeomorphic to $[-1,1] \times S^1$, and where $U_{\theta,\phi}$ is a disk about the point $(\theta,\phi)\in (S^1\times S^1) - \Delta_{S^1}$ that corresponds to a double point of $\beta$.


Each double point $P$ of $\beta$ determines a pair of points $(\theta,\phi)$ and $(\phi,\theta)$ in $S^1\times S^1 - \Delta_{S^1}$, where $\beta(\theta)= \beta(\phi)=P$. As in the introduction, $\beta'_P$ denotes the restriction of $\beta$ to the positively oriented arc in $S^1$ from $\theta$ to $\phi$, and $\beta''_P$ denotes its restriction to the arc from $\phi$ to $\theta$. Denote the initial tangent vectors of $\beta'_P$ and $\beta''_P$ by $\vv'$ and $\vv''$. The intersection number $\e_P$ is defined by
$$
\vv'\wedge\vv'' \in \e_P \times (\text{a positive number}) \times \text{ (the orientation of $M$ at $P$)}.
$$
An elementary computation shows that the intersection number of $\beta^2 : S^1\times S^1 \to M^2$ with $\Delta_M$ at $(\theta,\phi)$ is $-\e_P$, and is $\e_P$ at $(\phi,\theta)$. Consequently,
$$
\langle \w_\xi, Z'\rangle = -\e_P \text{ and } \langle \w_\xi, Z''\rangle = \e_P
$$
where $U'$ (resp.\ $U''$) denotes $U_{\theta,\phi}$ (resp.\ $U_{\phi,\theta}$) and $Z'$ (resp.\ $Z''$) is the positive generator of $H_2(U',\partial U';\Z)$ (resp.\ $H_2(U'',\partial U'';\Z)$). 

The contribution of the double point $P$ to $\langle \w_\xi, s_\beta \rangle$ is
thus
\begin{equation}
\label{eqn:P}
\langle \w_\xi, Z' \rangle\, \beta'_P \otimes \beta''_P + \langle \w_\xi, Z'' \rangle\, \beta''_P \otimes \beta'_P
=
-\e_P (\beta'_P \otimes \beta''_P - \beta''_P \otimes \beta'_P).
\end{equation}


It remains to compute the contribution of the strip $\Gamma$ to $\langle \w_\xi,s_\beta\rangle$. The derivative $\betadot : S^1 \to TM$ of $\beta$ corresponds to a section of the circle bundle $\partial N \to \Delta_M$, unique up to homotopy. By the construction preceding Lemma~\ref{lem:rot}, this determines a relative chain $\Gamma_\betadot$ in $(N,\Delta_M \cup \partial N)$.

The inverse image of $\Gamma$ in $[0,2\pi]\times S^1$ under the map (\ref{eqn:quot}) is the disjoint union of two strips, $\Gamma_0$, a regular neighbourhood of $0\times S^1$, and $\Gamma_{2\pi}$, a regular neighbourhood of $2\pi \times S^1$.

Give $\Gamma_0$ and $\Gamma_{2\pi}$ the orientation induced from $S^1\times S^1$. Then, as classes in $H_2(N,\Delta_M\cup\partial N)$, we have
$$
[\Gamma_0] = [\Gamma_\betadot] \text{ and } [\Gamma_{2\pi}] = -[\Gamma_{\betadot}].
$$
As observed in the proof of Proposition~\ref{prop:homologous}, the restriction of $s_\beta$ to $\Gamma_{2\pi}$ is homotopic to $1\otimes \beta$, and to $\Gamma_0$ is homotopic to $\beta\otimes 1$.

Lemma~\ref{lem:rot} now implies that the contribution to $\langle \w_\xi,s_\beta\rangle$ from $\Gamma$ is
\begin{multline}
\label{eqn:strip}
\langle \w_\xi,\Gamma\rangle = 
\langle \w_\xi,\Gamma_{2\pi} \rangle\, 1\otimes \beta
+ \langle \w_\xi,\Gamma_0\rangle \,\beta\otimes 1
= -\langle \w_\xi,\Gamma_\betadot \rangle\, 1\otimes \beta
+ \langle \w_\xi,\Gamma_\betadot\rangle\, \beta\otimes 1 \cr
= \rot_\xi(\beta)\big(\beta\otimes 1 - 1\otimes \beta\big).
\end{multline}
The result follows by adding the contribution of the strip (\ref{eqn:strip}) to the sum of the contributions (\ref{eqn:P}) of the double points $P$. 
\end{proof}

\begin{corollary}
\label{cor:formula}
If $\alpha \in \Lambda \Mhat$ and $\pi\circ \alpha$ is immersed in $M$ with transverse self intersections, then
$$
\varphi(\alpha) \cap c_\xi  = 
-\sum_P \e_P\big((\pi\circ \alpha_P')\otimes (\pi\circ \alpha'') - (\pi\circ \alpha'')\otimes (\pi\circ \alpha')\big).
$$
\end{corollary}

\begin{proof}
The class $\varphi(\alpha)$ is represented by $(s_{\pi\circ \alpha},\alphahat)$ and $\tau_\xi$ by $(\w_\xi,0)$.
Applying Proposition~\ref{prop:pairing} with $\beta = \pi\circ \alpha$, we obtain
\begin{multline*}
\varphi(\alpha) \cap c_\xi  = 
\rot_\xi(\pi\circ \alpha)\big((\pi\circ\alpha)\otimes 1 - 1 \otimes(\pi\circ\alpha) \big)
\cr
-\sum_P \e_P\big((\pi\circ \alpha_P')\otimes (\pi\circ \alpha'') - (\pi\circ \alpha'')\otimes (\pi\circ \alpha')\big).
\end{multline*}
The definition of the pairing $\cap$ implies that
$$
\varphi(\alpha)\cap f_\xi = \iota \langle f_\xi,\alpha\rangle = -\rot_\xi(\pi\circ\alpha) \big((\pi\circ\alpha)\otimes 1 - 1 \otimes(\pi\circ\alpha) \big).
$$
The result follows as $c_\xi = \tau_\xi + f_\xi$.
\end{proof}

\subsection{A homological description of $\delta_\xi$}

We can now give a homological description of the Turaev cobracket of a non-compact oriented surface $M$. Recall that, when $\bV$ is a local system over $M$, $H_0(M;\bV)$ is the maximal trivial quotient of $\bV$. Applying this when $\bV=\bL_M\otimes\bL_M$, we see that there is a canonical map
$$
H_0(M;\bL_M\otimes \bL_M) \to H_0(M;\bL_M)\otimes H_0(M;\bL_M) \cong H_0(\Lambda M)\otimes H_0(\Lambda M).
$$
For a section $\xi$ of $\Mhat \to M$, define
$$
p_\xi : H_2(M^2,\Mhat;\bL_\Mhat\to \bP_M\otimes\bP_M^\op) \to H_0(M;\bL_M) \otimes H_0(M;\bL_M) \cong H_0(\Lambda M)^{\otimes 2}
$$
to be the composite
$$
\xymatrix@C=16pt{
H_2(M^2,\Mhat;\bL_\Mhat\to \bP_M\otimes\bP_M^\op) \ar[r]^(.6){\blank \cap c_\xi} & H_0(M;\bL_M\otimes\bL_M) \ar[r] & H_0(\Lambda M) \otimes H_0(\Lambda M).
}
$$

Each section $\xi : M \to \Mhat$ of $\pi$ induces a map $\Lambda\xi : \Lambda M \to \Lambda \Mhat$ and thus a homomorphism
$$
(\Lambda\xi)_\ast : H_0(\Lambda M) \to H_0(\Lambda \Mhat).
$$
It is injective as its composition with $(\Lambda \pi)_\ast$ is the identity. The image of a free homotopy class of $f : S^1 \to M$ corresponds to the regular homotopy class of an immersed circle $\alpha$ with $\rot_\xi(\alpha) = 0$ that is freely homotopic to $f$.

The following factorization of $\delta_\xi$ follows directly from Corollary~\ref{cor:formula}.

\begin{theorem}
\label{thm:factor_turaev}
If $M$ is a non-compact oriented surface and $\xi$ is a section of $\pi : \Mhat \to M$, then the diagram
$$
\xymatrix{
H_0(\Lambda M) \ar[r]^{(\Lambda\xi)_\ast} \ar[d]_{\delta_\xi} \ar[r] & H_0(\Lambda \Mhat) \ar[r]^\cs \ar[dr]^\varphi & H_1(\Lambda \Mhat)
\cr
H_0(\Lambda M) \otimes H_0(\Lambda M) && H_2(M^2,\Mhat;\bL_\Mhat\to \bP_M\otimes\bP_M^\op) \ar[ll]^(.55){-p_\xi} \ar[u]^\cong_\psi
}
$$
commutes.
\end{theorem}

\section{De~Rham Aspects}
\label{sec:DR}

In this section, in preparation for applying the machinery of Hodge theory in Section~\ref{sec:hodge_turaev}, we construct de~Rham versions of the continuous duals of the maps used in the homological description of the Turaev cobracket given in Section~\ref{sec:turaev}.

\subsection{Preliminaries}

Suppose that $N$ is a smooth manifold with finite first Betti number and that $\kk$ is a field of characteristic zero. We are especially interested in the case where $N$ is a rational $K(\pi,1)$ space.

Recall from \cite[\S7]{hain:goldman} that $H_0(P_{x_0,x_1}N;\kk)$ and $H_0(\Lambda N;\kk)$ have natural topologies and that their continuous duals are denoted
$$
\Hdual^0(P_{x_0,x_1}N;\kk) := \Hom^\cts_\kk(H_0(P_{x_0,x_1}N),\kk)
$$
and
$$
\Hdual^0(\Lambda N;\kk) := \Hom^\cts_\kk(H_0(\Lambda N),\kk).
$$
Recall from \cite[\S8]{hain:goldman} that $\Ldual_N$ denotes the continuous dual of the local system $\bL_N$. There is a natural isomorphism \cite[Thm.~6.9]{hain:goldman}.
$$
\Hdual^0(\Lambda N;\kk) \cong H^0(N;\Ldual_N).
$$

Denote the local system over $N\times N$ whose fiber over $(x_0,x_1)$ is $\Hdual^0(P_{x_0,x_1}N;\kk)$ by $\Pdual_N$ and its pullback along the interchange map $N^2 \to N^2$ by $\Pdual_N^\op$. 

\begin{lemma}
Let $p:N\times N\to N$ be projection onto the first factor. If $N$ is a rational $K(\pi,1)$, then there is a natural isomorphism of locally constant sheaves
$$
R^k p_\ast (\Pdual_N\otimes\Pdual_N^\op) \cong
\begin{cases}
\Ldual_N & k=0,\cr
0 & k\neq 0
\end{cases}
$$
over $N$.
\end{lemma}

\begin{proof}
This follows directly from \cite[Cor.~9.2]{hain:goldman}.
\end{proof}

\begin{corollary}
\label{cor:coho_iso}
If $N$ is a rational $K(\pi,1)$, then there is a natural isomorphism
$$
H^j(N^2;\Pdual_N\otimes \Pdual_N^\op) \cong H^j(N;\Ldual_N).
$$
\end{corollary}

\begin{proof}
Apply the Leray spectral sequence of the projection $p:N\times N \to N$. The previous result and the fact that $N$ is a rational $K(\pi,1)$ imply that
$$
E_2^{j,k} \cong 
\begin{cases}
H^j(N;\Ldual_N) & k=0,\cr
0 & k>0
\end{cases}
$$
so that the spectral sequence collapses at $E_2$.
\end{proof}

\subsubsection{Differential forms}

Now $\kk$ will be $\R$ or $\C$. We regard a local system on $N$ as a locally constant sheaf. We will denote the complex of differential forms on $N$ with values in a local system $\bV$ of real (or rational) vector spaces by $E^\bdot(N;\bV)$. In \cite{hain:goldman}, we denoted the flat vector bundle associated to $\bV$ by $\sV$ and the sheaf of $j$-forms on $N$ with values in $\bV$ by $\sE^j_N\otimes\sV$. So $E^j(N,\bV)$ is just the space of global sections of $\sE^j_N\otimes\sV$. There are therefore isomorphisms
$$
H^\bdot(E^\bdot(N;\bV)) \cong H^\bdot(N;\bV)
$$
To connect with \cite{hain:goldman}, we point out that the flat vector bundle associated to $\Ldual_N$ is denoted by $\sL_N$, and the flat vector bundle associated to $\Pdual_N$ by $\sP_N$.

\subsection{Continuous DR duals}
\label{sec:cts_duals}

In this section, $M$ is an oriented surface of non-positive Euler characteristic and $\pi : \Mhat \to M$ is the bundle of non-zero tangent vectors of $M$. Both $M$ and $\Mhat$ are rational $K(\pi,1)$ spaces.\footnote{For $M$ this is proved in \cite[\S5.1]{hain:goldman}. That $\Mhat$ is also a rational $K(\pi,1)$ follows from this using the fact that an oriented circle bundle over a rational $K(\pi,1)$ is a rational $K(\pi,1)$.}

\subsubsection{The continuous dual of $H_\bdot(M^2,\Mhat;\bL_\Mhat\to \bP_M\otimes\bP_M^\op)$}

As in Section~\ref{sec:turaev}, we denote the composition of the projection $\pi$ with the diagonal map $M\to M^2$ by $\Deltabar$. There is a natural restriction mapping
$$
\iota^\ast : \Deltabar^\ast (\Pdual_M \otimes \Pdual_M^\op) \to \Ldual_\Mhat
$$
dual to the map (\ref{eqn:iota}). Its restriction
$$
\Hdual^0(\Lambda_x M) \otimes \Hdual^0(\Lambda_x M) \to \Hdual^0(\Lambda_v \Mhat)
$$
to the fiber over $v\in \Mhat$, where $x = \pi(v)$, is
$$
f \otimes g \mapsto f(1_x)\otimes (\pi^\ast g) - (\pi^\ast f)\otimes g(1_x).
$$

Since $\Pdual_M$ and $\bL_\Mhat$ are local systems of algebras, 
$\Deltabar$ and $\iota$ induce a DGA homomorphism
$$
\Deltabar^\ast\otimes \iota^\ast : E^\bdot(M^2;\Pdual_M\otimes\Pdual_M^\op) \to E^\bdot(\Mhat,\Ldual_\Mhat).
$$
Define
$$
E^\bdot(M^2,\Mhat;\Pdual_M\otimes\Pdual_M^\op\to \Ldual_\Mhat) := \cone\big(E^\bdot(M^2;\Pdual_M\otimes\Pdual_M^\op) \overset{\Deltabar^\ast\otimes \iota^\ast}{\To} E^\bdot(\Mhat,\Ldual_\Mhat)\big)[-1].
$$
Denote its cohomology groups by
$$
H^\bdot(M^2,\Mhat;\Pdual_M\otimes\Pdual_M^\op \to \Ldual_\Mhat).
$$
These cohomology groups are dual to the homology of the cone defined in Section~\ref{sec:turaev}.

\begin{proposition}
The pairing
$$
\bil : E^\bdot(M^2,\Mhat;\Pdual_M\otimes\Pdual_M^\op\to \Ldual_\Mhat) \otimes C_\bdot(\Deltabar_\ast\otimes\iota) \to \kk
$$
$$
\langle (\w,\xi),(s,u)\rangle = \int_s\w + \int_u \xi
$$
defined using integration and the pairings
$$
\Pdual_M \otimes \bP_M \to \kk \text{ and } \Ldual_\Mhat \otimes \bL_\Mhat \to \kk
$$
respects the differentials and thus induces a pairing
$$
\bil : H^\bdot(M^2,\Mhat;\Pdual_M\otimes\Pdual_M^\op \to \Ldual_\Mhat) \otimes
H_\bdot\big(M^2,\Mhat;\bL_\Mhat\to \bP_M\otimes\bP_M^\op\big) \to \kk. \qed
$$
\end{proposition}

\begin{proposition}
\label{prop:psidual}
If $M$ is an oriented non-compact surface, then there is a natural isomorphism
$$
\psidual : H^1(\Mhat;\Ldual_\Mhat) \overset{\simeq}{\To} H^2(M^2,\Mhat;\Pdual_M\otimes\Pdual_M^\op \to \Ldual_\Mhat)
$$
that is dual to the isomorphism $\psi$ in Lemma~\ref{lem:les_cone} with respect to the pairing $\bil$.
\end{proposition}

\begin{proof}
The cohomology long exact sequence of the cone is
\begin{multline*}
\cdots \to 
H^1(M^2;\Pdual_M\otimes\Pdual_M^\op) \overset{\partialdual}{\to} H^1(\Mhat;\Ldual_\Mhat)
\cr
\overset{\psidual}{\to}
H^2(M^2,\Mhat;\Pdual_M\otimes\Pdual_M^\op \to \Ldual_\Mhat)
\to H^2(M^2;\Pdual_M\otimes\Pdual_M^\op) \overset{\partialdual}{\to} \cdots
\end{multline*}
It is dual to the long exact sequence of the cone $C_\bdot(\Deltabar\otimes \iota)$ in Lemma~\ref{lem:les_cone}.

Since $M$ is a non-compact surface, it is homotopy equivalent to a wedge of circles and therefore a rational $K(\pi,1)$ of cohomological dimension 1. In particular, $H^2(M^2;\Pdual_M\otimes\Pdual_M^\op)$ vanishes. Finally, Corollary~\ref{cor:coho_iso} gives an isomorphism $H^1(M^2;\Pdual_M\otimes\Pdual_M^\op) \cong H^1(M;\Ldual_M)$. As in the proof of Lemma~\ref{lem:les_cone}, the connecting homomorphism $\partialdual$ vanishes, which implies that $\psidual$ is an isomorphism.
\end{proof}

Recall \cite[Prop.~8.1]{hain:goldman} that there is a map
$$
\csdual : H^1(\Mhat;\Ldual_\Mhat) \to \Hdual^0(\Lambda\Mhat).
$$
dual to the Chas--Sullivan map $\cs : H_0(\Lambda \Mhat) \to H_1(\Mhat,\bL_\Mhat)$ under $\bil$.

\begin{corollary}
If $M$ is an oriented non-compact surface, there is a map
$$
\phidual : H^2(M^2,\Mhat;\Pdual_M\otimes\Pdual_M^\op \to \Ldual_\Mhat) \to \Hdual^0(\Lambda \Mhat)
$$
that is dual to $\varphi$ under the pairing $\bil$ and corresponds to $\csdual$ in the sense that $\csdual=\phidual\circ\psidual$. \qed
\end{corollary}

\subsubsection{The cup product}

The de~Rham incarnation of the complex $C_\Delta^\bdot(M^2,\Mhat)$ defined in Section~\ref{sec:pairing} is
$$
E^\bdot_\Delta(M^2,\Mhat) := \cone\big(E^\bdot(M^2) \to E^\bdot(M^2-\Delta) \oplus E^\bdot(\Mhat)\big)[-1].
$$
De~Rham's Theorem and the 5-lemma imply that it computes $H^\bdot_\Delta(M^2,\Mhat;\kk)$.

\begin{lemma}
There is a well-defined product
\begin{equation}
\label{eqn:cup}
\smile\ : H^0(M;\Ldual_M\otimes\Ldual_M)\otimes H^2_\Delta(M^2,\Mhat) \to
H^2(M^2,\Mhat;\Pdual_M\otimes\Pdual_M^\op \to \Ldual_\Mhat).
\end{equation}
It is dual to the pairing
$$
\bil :
H_2(M^2,\Mhat;\bL_\Mhat\to \bP_M\otimes\bP_M^\op)\otimes H^2_\Delta(M^2,\Mhat) \to H_0(M;\bL_M\otimes \bL_M).
$$
of Proposition~\ref{prop:pairing} in the sense that
$$
\big\langle f \smile c, z \big\rangle = \big\langle f,\langle z,c \rangle\big\rangle
$$
for all
$$
f\in H^0(M;\Ldual_M\otimes\Ldual_M),\ c\in H^2_\Delta(M^2,\Mhat),\ z\in H_2(M^2,\Mhat;\bL_\Mhat\to \bP_M\otimes\bP_M^\op).
$$
\end{lemma}

\begin{proof}
This result can be proved using differential forms or singular cochains. We will use differential forms. The proof using singular cochains is similar.

Choose regular neighbourhoods $U$ and $V$ of the diagonal $\Delta$ in $M^2$, where $V \subset U$, $V$ is closed and $U$ is open. Since $\Delta \hookrightarrow U$ is a homotopy equivalence, every flat section of $\Ldual_M\otimes \Ldual_M$ over the diagonal extends uniquely to a flat section of $\Pdual_M\otimes \Pdual_M^\op$ over $U$. It follows that restriction to the diagonal induces a quasi-isomorphism
$$
E^\bdot(U;\Pdual_M \otimes \Pdual_M^\op) \to E^\bdot(M,\Ldual_M\otimes\Ldual_M).
$$

Since the inclusion $\Delta \to V$ is a homotopy equivalence, the map
$$
E^\bdot_\Delta(M^2) \to E^\bdot_V(M^2) := \cone\big(E^\bdot(M^2) \to E^\bdot(M^2-V)\big)[-1]
$$
is a quasi-isomorphism. Denote the complex of forms of $M^2$ that vanish on $M^2-V$ by $E^\bdot(M^2,M^2-V)$. The 5-lemma implies that the cochain map
$$
E^\bdot(M^2,M^2-V) \to E^\bdot_V(M^2)
$$
that takes $\w$ to $[\w,0]$ is a quasi-isomorphism. Together these imply that $E^\bdot_\Delta(M^2,\Mhat)$ is quasi-isomorphic to the complex
$$
\cone\big(E^\bdot(M^2,M^2-V) \to E^\bdot(\Mhat)\big)[-1].
$$

The cup product pairing (\ref{eqn:cup}) is induced by the map of complexes
\begin{multline*}
E^\bdot(U,\Pdual_M\otimes \Pdual_M^\op) \otimes \cone\big(E^\bdot(M^2,M^2-V) \to E^\bdot(\Mhat)\big)[-1]
\cr
\to
E^\bdot(M^2,\Mhat;\Pdual_M\otimes \Pdual_M^\op \to \Ldual_\Mhat)
\end{multline*}
defined by
$
F \otimes [\w,\eta] \mapsto [F \wedge \w,(-1)^{|F|}(\pi^\ast F) \wedge \eta].
$
This is a chain map according to the conventions in Section~\ref{sec:cones}.

To prove the remaining assertion, suppose that $z$ is represented by $[s,u]$ in $C_2(\Deltabar_\ast\otimes\iota)$, $f$ is represented by $F \in E^0(U;\Pdual_M \otimes \Pdual_M^\op)$, and $c$ is represented by $[\w,\eta] \in \cone\big(E^\bdot(M^2,M^2-V) \to E^\bdot(\Mhat)\big)[-1]$. Then, $f\smile c$ is represented by $[F\w,\pi^\ast F \cdot \eta]$.
and
$$
\langle f\smile c, z \rangle = \langle [F\w,\pi^\ast F \cdot\eta],[s,u] \rangle = \int_s F\w + \int_u F\eta.
$$
On the other hand, since $F$ is locally constant,
$$
\langle f,\langle z,c \rangle \rangle
= \langle f, \langle [\w,\eta],[s,u] \rangle\rangle
= \langle F, \int_s \w + \int_u \eta \rangle
= \int_s F\w + \int_u F\eta.
$$
\end{proof}

\subsection{Factorization of the continuous dual of the Turaev cobracket}
\label{sec:factorization}

Define
$$
\deltadual_\xi : \Hdual^0(\Lambda M)\otimes\Hdual^0(\Lambda M) \to \Hdual^0(\Lambda M)
$$
so that the diagram
$$
\xymatrix@C=16pt{
\Hdual^0(\Lambda M)^{\otimes 2} \ar[r]^(.32)\simeq \ar@{.>}[d]_{\deltadual_\xi} &
H^0(M;\Ldual_M) \otimes H^0(M;\Ldual_M) \ar[r]^(.55){\text{mult}} &
H^0(M;\Ldual_M\otimes\Ldual_M) \ar[d]^{\phantom{x}\smile c_\xi}
\cr
\Hdual^0(\Lambda M) & \Hdual^0(\Mhat;\Ldual_\Mhat) \ar[l]_{(\Lambda\xi)^\ast} &
H^2(M^2,\Mhat;\Pdual_M\otimes\Pdual_M^\op \to \Ldual_\Mhat) \ar[l]_(.6)\phidual
}
$$
commutes. The next result follows directly from Theorem~\ref{thm:factor_turaev} and the results in Section~\ref{sec:cts_duals}.

\begin{proposition}
The map $\deltadual_\xi$ is the continuous dual of $\delta_\xi$ in the sense that
$$
\langle \deltadual_\xi(f\otimes g),\alpha\rangle = \langle f\otimes g,\delta_\xi(\alpha)\rangle
$$
for all $f,g\in \Hdual^0(\Lambda M)$ and $\alpha \in \Lambda M$.
\end{proposition}

\section{Proof of Theorem~\ref{thm:turaev}}
\label{sec:hodge_turaev}

In this section, $\kk$ will be $\Q$, $\R$ or $\C$, as appropriate, and $X$ will be a smooth affine curve over $\C$. Equivalently, $X$ is the complement $\Xbar - D$ of a finite subset $D$ of a compact Riemann surface $\Xbar$. Denote the holomorphic tangent bundle of $\Xbar$ by $T\Xbar$.

\subsection{The map $(\Lambda\xi)^\ast$ is a morphism of MHS}

\begin{definition}
\label{def:q-alg}
Suppose that $m$ is a positive integer. An {\em algebraic $m$-framing} of $X$ is a meromorphic section of $L\otimes T\Xbar$ whose divisor is supported on $D$, where $L$ is a holomorphic line bundle over $\Xbar$ whose $m$th power $L^{\otimes m}$ is trivial. Equivalently, $\xi$ is the $m$th root of a meromorphic section of the $m$th power of the holomorphic tangent bundle of $\Xbar$ whose divisor is supported on $D$. A {\em quasi-algebraic framing} of $X$ is an algebraic $m$-framing for some $m>0$. An {\em algebraic framing} of $X$ is, by definition, a 1-framing.
\end{definition}

Since torsion line bundles on $\Xbar$, such as $L$, are topologically trivial, each quasi-algebraic framing of $X$ determines a homotopy class of smooth framings of $X$ and a cobracket $\delta_\xi$. In this section, we prove the following stronger version of Theorem~\ref{thm:turaev}.

\begin{theorem}
\label{thm:strong_turaev}
If $\xi$ is a quasi-algebraic framing of $X$, then
$$
\delta_\xi : \Q\lambda(X)^\wedge\otimes\Q(-1) \to \Q\lambda(X)^\wedge\comptensor \Q\lambda(X)^\wedge
$$
is a morphism of pro-mixed Hodge structures.
\end{theorem}

Throughout this section, $m$ is a fixed positive integer, and $\xi$ is an algebraic $m$-framing of $X$. Its $m$th power $\xi^m$ is a meromorphic section of $(T\Xbar)^{\otimes m}$. The theorem is proved by showing that each group in the factorization of
$$
\deltadual_\xi : \Hdual^0(\Lambda X) \otimes\Hdual^0(\Lambda X) \to  \Hdual^0(\Lambda X)\otimes \Q(-1)
$$
given in Section~\ref{sec:factorization} has a mixed Hodge structure (MHS) and that each morphism in the factorization is a morphism of MHS. The twist by $\Q(-1)$ occurs in the map $\smile c_\xi$. Note that the topological factorization of $\delta_\xi$ in Section~\ref{sec:turaev} implies that all of the maps in the factorization of $\deltadual_\xi$ in Section~\ref{sec:factorization} are also defined over $\Q$. So we need only show that each preserves the Hodge and weight filtrations after extending scalars to $\C$.

For a positive integer $d$, denote the set of non-zero elements of $(TX)^{\otimes d}$ by $\Xhat_d$. This is a smooth quasi-projective variety. The map $TX \to (TX)^{\otimes d}$ that takes a tangent vector $v$ to $v^d$ induces a covering map $p_d : \Xhat \to \Xhat_d$. Since $X$, $\Xhat_d$ are smooth algebraic varieties, $\Hdual^0(\Lambda \Xhat_d)$ and $\Hdual^0(\Lambda X)$ have natural MHS by \cite[Cor.~10.7]{hain:goldman}.

\begin{lemma}
For all $d\ge 1$, the map $(\Lambda p_d)^\ast : \Hdual^0(\Lambda \Xhat_d;\Q) \to \Hdual^0(\Lambda \Xhat;\Q)$ is an isomorphism of MHS.
\end{lemma}

\begin{proof}
Since $p_d : \Xhat \to \Xhat_d$ is a morphisms of varieties, $(\Lambda p_m)^\ast$ is a morphism of MHS. So, to prove the result, it suffices to prove that it is an isomorphism of vector spaces.

To this end, fix a smooth section $\xi_o$ of $\Xhat \to X$. Then for all $k\ge 1$, $\xi_o^k$ is a smooth section of $\Xhat_k \to X$. Since $\Xhat_k \to X$ is a principal $\C^\ast$-bundle with section $\xi_o^k$, it is trivialized by the map
\begin{equation}
\label{eqn:triv}
\phi_k : X \times \C^\ast \to \Xhat_k,\qquad (x,t) \mapsto t\xi_o(x)^k.
\end{equation}
This trivialization induces an isomorphism
$$
(\Lambda \phi_k)^\ast : \Hdual^0(\Lambda \Xhat_k;\Q) \to \Hdual^0(\Lambda(X \times \C^\ast)) \cong \Hdual^0(\Lambda X;\Q)\otimes \Hdual^0(\Lambda\C^\ast;\Q)
$$
as there is a canonical isomorphism $\Lambda(A\times B) \cong \Lambda A\times\Lambda B$.

The $d$-fold covering map $\chi_d : \C^\ast \to \C^\ast$ induces an isomorphism $\Hdual^0(\Lambda\C^\ast;\Q)\to \Hdual^0(\Lambda\C^\ast;\Q)$. Since the diagram
$$
\xymatrix{
X \times \C^\ast \ar[r]^(.55){\phi_1} \ar[d]_{\id \times \chi_d} & \Xhat \ar[d]^{p_d} \cr
X \times \C^\ast \ar[r]^(.55){\phi_d} & \Xhat_d
}
$$
commutes, so does
$$
\xymatrix@C=32pt{
\Hdual^0(\Lambda\Xhat_d;\Q)\ar[d]_{(\Lambda\p_d)^\ast}  \ar[r]^(.38){(\Lambda \phi_d)^\ast} & \Hdual^0(\Lambda X;\Q)\otimes \Hdual^0(\Lambda\C^\ast;\Q) \ar[d]^{\id\otimes \chi_d}\cr
\Hdual^0(\Lambda\Xhat;\Q) \ar[r]^(.38){(\Lambda \phi_1)^\ast} & \Hdual^0(\Lambda X;\Q)\otimes \Hdual^0(\Lambda\C^\ast;\Q).
}
$$
The result follows as the two horizontal maps and the right-hand vertical map are isomorphisms.
\end{proof}

When $m>1$, it is not immediately obvious that $(\Lambda \xi)^\ast$ is a morphism of MHS. However, this is the case.

\begin{corollary}
The map $(\Lambda \xi)^\ast : \Hdual^0(\Lambda \Xhat) \to \Hdual^0(\Lambda X)$
is a morphism of MHS.
\end{corollary}

\begin{proof}
Regard $\xi$ as a smooth section of $\Xhat\to X$. Since $\xi^m$ is homotopic to $p_m\circ \xi$, the diagram
$$
\xymatrix@C=32pt{
\Hdual^0(\Lambda \Xhat_m) \ar[r]^{(\Lambda \xi^m)^\ast}\ar[d]_{(\Lambda p_m)^\ast}^{\cong} & \Hdual^0(\Lambda X;\Q) \ar@{=}[d] \cr
\Hdual^0(\Lambda \Xhat) \ar[r]^{(\Lambda \xi)^\ast} & \Hdual^0(\Lambda X;\Q)
}
$$
commutes. Since $\xi^m$ is algebraic, the map $(\Lambda \xi^m)^\ast$ is a morphism of MHS. The result follows as $(\Lambda p_m)^\ast$ is an isomorphism of MHS by the previous result.
\end{proof}

\subsection{The multiplication map is a morphism of MHS}

Each fiber $\Hdual^0(\Lambda_x X)$ of $\Ldual_X$ is a commutative Hopf algebra in $\MHS$. The product induces a map $\Ldual_X \otimes \Ldual_X \to \Ldual_X$ of local systems. It is a direct limit of morphisms of admissible variations of MHS over $X$. The Theorem of the Fixed Part (or a direct argument that uses the construction of these MHS) implies that
$$
\text{mult}\ : H^0(X,\Ldual_X)^{\otimes 2} \to H^0(X,\Ldual_X^{\otimes 2})
$$
is a morphism of MHS.

\subsection{The map $\phidual$ is a morphism of MHS}

To prove that the remaining groups have natural MHS and that the maps between them are morphisms, we need to recall the following standard fact about cones of mixed Hodge complexes, which is implicit in \cite{deligne:hodge3}.

\begin{lemma}
\label{lem:cone}
The cone $C^\bdot(\phi)$ of a morphism $\phi : B^\bdot \to A^\bdot$ of mixed Hodge complexes is a mixed Hodge complex, and the corresponding long exact sequence
$$
\cdots \to H^{j-1}(A^\bdot) \to H^j(C^\bdot(\phi)) \to H^j(B^\bdot) \to H^j(A^\bdot) \to \cdots
$$
is a long exact of MHS. \qed
\end{lemma}

\begin{proposition}
\label{prop:phi}
If $X$ is an affine curve, then $H^2(X^2,\Xhat;\Pdual_X\otimes\Pdual_X^\op \to \Ldual_\Xhat)$ has a natural MHS and $\psidual : H^1(X;\Ldual_\Xhat) \to H^2(X^2,\Xhat;\Pdual_X\otimes\Pdual_X^\op \to \Ldual_\Xhat)$ is an isomorphism of MHS. Consequently,
$$
\phidual : H^2(X^2,\Xhat;\Pdual_X\otimes\Pdual_X^\op \to \Ldual_\Xhat) \to \Hdual^0(\Lambda \Xhat)
$$
is also a morphism of MHS.
\end{proposition}

\begin{proof}
The work of Saito \cite{saito:mhc} implies that if $\bV$ is an admissible variation of MHS over the complement of a divisor $W$ with normal crossings in a smooth variety $Z$, then the complex $E^\bdot(Z\log W;\bV)$ of smooth forms on $Z$ with values in the canonical extension of $\bV$ to $Z$ and log poles along $W$ is part of a mixed Hodge complex and is naturally quasi-isomorphic to $E^\bdot(Z-W;\bV)$. In particular, it computes $H^\bdot(Z-W;\bV)\otimes\C$, together with its Hodge and weight filtrations.

The compactification $P=\P(T\Xbar\oplus\cO_\Xhat)$ of the tangent bundle $T\Xbar$ of $\Xbar$ is a compactification of $\Xhat$ whose complement $W$ is a divisor with normal crossings. The cone
$$
\cone\big(E^\bdot(\Xbar^2,\log((\Xbar\times D) \cup (D\times\Xbar));\Pdual_X\otimes\Pdual_X^\op),E^\bdot(P\log W;\bL_\Xhat)\big)[-1]
$$
is quasi-isomorphic to $E^\bdot(X^2,\Xhat;\Pdual_X\otimes\Pdual_X^\op \to\Ldual_\Xhat)$. Lemma~\ref{lem:cone} implies that it is the complex part of a mixed Hodge complex that computes $H^\bdot(X^2,\Xhat;\Pdual_X\otimes\Pdual_X^\op \to \Ldual_\Xhat)$. In particular these cohomology groups have a natural MHS. It also follows that $\psidual$, being a map in cohomology sequence, is a morphism (and thus isomorphism) of MHS.

The map $\csdual$ is morphism of MHS by \cite[Lem.~11.1]{hain:goldman}. Since $\phidual = \csdual\circ \psidual^{-1}$, it is also a morphism of MHS.
\end{proof}

\subsection{The class $c_\xi$ is a Hodge class}

\begin{proposition}
The group $H_\Delta^\bdot(X^2,\Xhat)$ has a natural mixed Hodge structure and $c_\xi$ is a Hodge class of type $(1,1)$.
\end{proposition}

\begin{proof}
Let $Y$ be the blow up $\Xbar\times \Xbar$ at $\Delta_D$. Then $X^2-\Delta$ is the complement of a normal crossing divisor $E$ in $Y$. Write $E=E'+\Delta_\Xbar$. The restriction of $E'$ to the diagonal $\Delta_\Xbar$ is $\Delta_D$.

Let $Z$ be the normal crossings compactification of $\Xhat$ constructed in the proof of Proposition~\ref{prop:phi}. The commutative diagram of morphisms of complex algebraic maps
$$
\xymatrix{
\Xhat \ar[r]^\pi\ar@{=}[d] & X \ar[r]^\Delta\ar@{=}[d] & X^2\ar@{=}[d] & X^2 - \Delta_X \ar[l]\ar@{=}[d]\cr
Z-W \ar[r] & \Xbar - D \ar[r] & Y-E' & Y-E \ar[l]
}
$$
induces a commutative diagram
$$
\xymatrix{
E^\bdot(Y\log E) \ar[d] & E^\bdot(Y\log E') \ar[l]\ar[r]\ar[d] & E^\bdot(\Xbar\log D) \ar[r]\ar[d] & E^\bdot(Z\log W) \ar[d] \cr
E^\bdot(X^2-\Delta_X) & E^\bdot(X^2) \ar[l] \ar[r]^{\Delta^\ast} & E^\bdot(X) \ar[r]^{\pi^\ast} & E^\bdot(\Xhat)
}
$$
of DGAs in which each vertical map is a quasi-isomorphism. Each DGA in this diagram is the complex part of the natural mixed Hodge complex associated to the corresponding variety. The Five Lemma implies that the complex $E^\bdot_\Delta(X^2,\Xbar)$ is naturally quasi-isomorphic to
\begin{equation}
\label{eqn:hodge_cone}
\cone\big(E^\bdot(Y\log E') \to E^\bdot(Y\log E) \oplus E^\bdot(Z\log W) \big)[-1]
\end{equation}
Lemma~\ref{lem:cone} implies that it is the complex part of a mixed Hodge complex. It follows that $H^\bdot_\Delta(X,\Xhat)$ has a natural MHS and that the exact sequence of Lemma~\ref{lem:s}
$$
0 \to H^1(\Xhat) \to H^2_\Delta(X^2,\Xhat) \to H^2_\Delta(X^2) \to 0
$$
is an exact sequence of mixed Hodge structures.

It remains to show that $c_\xi$ is a Hodge class that spans a copy of $\Q(-1)$. Recall the notation and the construction of $c_\xi$ from Section~\ref{sec:c}. In particular, $c_\xi = \pi^\ast\tau_\xi + f_\xi$. Since $\pi^\ast : H^2_\Delta(X^2,\Delta_X) \to  H^2_\Delta(X^2,\Xhat)$ is a morphism of MHS, to prove that $c_\xi$ is a Hodge class, it suffices to prove that both $f_\xi$ and $\tau_\xi$ are Hodge classes.

We first show that $f_\xi$ is a Hodge class. Let $r_m : \Xhat \to \C^\ast$ be the composite
$$
\xymatrix{
\Xhat \ar[r]^{p_m} & \Xhat_m \ar[r] & \C^\ast
}
$$
where the second map is the composition of the inverse of the isomorphism $\phi_m$ (\ref{eqn:triv}) with the projection $X\times \C^\ast \to \C^\ast$ given by $\xi^m$. Since $\xi^m$ is algebraic, this is a morphism of varieties and thus induces a morphism of MHS on cohomology. The map $r$ used in Lemma~\ref{lem:f_xi} in the construction of $f_\xi$ is the topological $m$th root of $r_m$. Since $r_m^\ast dt/t =  m r^\ast dt/t$, we have
$$
2 \pi i f_\xi = r^\ast \frac{dt}{t} = \frac{1}{m}r_m^\ast \frac{dt}{t} \in H^1(\Xhat;\Q) \subset H^2_\Delta(X^2,\Xhat)
$$
which spans a copy of $\Q(-1)$ as $H^1(\C^\ast;\Q) \cong \Q(-1)$. Thus $f_\xi$ is a Hodge class.

Since $\tau_\xi \in H^2_\Delta(X^2,X;\Q)$, to prove that it is a Hodge class, it suffices to show that it is a {\em real} Hodge class. To do this, we use the fact that the MHS on $H_\Delta^\bdot(X^2,\Delta_X)$ depends only on $X$ and the normal bundle of $\Delta_X$ in $X^2$, which is just the (holomorphic) tangent bundle $TX$ of $X$. This follows from the construction of a (real) mixed Hodge complex for the punctured neighbourhood of one variety in another that was constructed in \cite{durfee-hain}. That construction implies that the natural isomorphism
$$
H^\bdot_X(TX,X) \cong H^\bdot_\Delta(X^2,\Delta_X)
$$
that is constructed using topology, is an isomorphism of real MHS. There is also a natural isomorphism
$$
p_d^\ast : H^\bdot_X((TX)^{\otimes d},X) \overset{\simeq}{\To} H^\bdot_X(TX,X)
$$
of MHS for all $d\ge 1$, where $H^\bdot_X((TX)^{\otimes d},X)$ is defined to be the homology of the complex
$$
\cone\big(C^\bdot((TX)^{\otimes d},\Xhat_d) \to C^\bdot(X)\big)[-1].
$$

The trivialization $\phi_m : \Xhat_m \to X\times \C^\ast$ given by $\xi^m$ induces a MHS morphism $\phi_m^\ast : H^2(\C,\C^\ast) \to H^2_X((TX)^{\otimes m},X)$. The class $\tau_\xi$ is the image of the positive generator $\tau_B$ of $H^2(\C,\C^\ast) \cong \Z(-1)$ under the sequence
$$
\xymatrix@C=16pt{
H^2(\C,\C^\ast) = H^2_{\{0\}}(\C) \ar[r]^(.55){r_m^\ast} & H^2_X((TX)^{\otimes m},X) \ar[r]^(.55){p_m^\ast} & H^2_X(TX,X) & \ar[l]_{\simeq} H^2_\Delta(X^2,\Delta_X).
}
$$
It follows that $\tau_\xi$ is a real (and therefore rational) Hodge class, which completes the proof.
\end{proof}

\begin{corollary}
The cup product (\ref{eqn:cup}) is a morphism of MHS. Consequently, cupping with $c_\xi$
$$
\smile c_\xi : \Hdual^0(X;\Ldual_X \otimes \Ldual_X) \to H^2(X^2,\Xhat;\Pdual_X\otimes\Pdual_X^\op \to \Ldual_\Xhat)\otimes \Q(-1)
$$
is a morphism of MHS
\end{corollary}

\section{Mapping Class Group Orbits of Framings}
\label{sec:framings}

In this section, we recall Kawazumi's classification \cite{kawazumi:framings} of mapping class group orbits of framings of a surface. As we shall see subsequently, this classification is closely related to the classification of the strata of meromorphic 1-forms studied by Kontsevich and Zorich \cite{kontsevich-zorich} in the holomorphic case, and by Chen, Gendron, Grushevsky and M\"oller \cite{mero_diffs} in the meromorphic case.

We first recall the definition of mapping class groups and our notation for them. Suppose that $Q$ is a finite subset of $\Sbar$ with $\#Q=m$, and $V$ is a set of $r$ non-zero tangent vectors that are anchored at $r$ distinct points, none of which are in $Q$. The mapping class group $\G_{g,m+\vec{r}}$ is defined to be the group $\pi_0\Diff^+(\Sbar,Q,V)$ of isotopy classes of $\Sbar$ that fix the points $Q$ and the tangent vectors $V$. The indices $m$ and $r$ are omitted when they vanish.

Suppose that $Q$ is non-empty. Set $S=\Sbar - Q$. The mapping class group $\G_{g,m}$ acts on framings of $S$ by push forward. Kawazumi \cite{kawazumi:framings} determined the set of mapping class group orbits. They depend on the vector $\dd(\xi) = (d_q)_{q\in Q} \in \Z^Q$ of local degrees of $\xi$ at the points of $\Q$. We say that $\dd(\xi)$ is {\em even} if each $d_q$ is even. When $g>0$ and $\dd(\xi)$ is even, we can associate the $\F_2$ quadratic form 
$$
F_\xi : H_1(\Sbar;\F_2) \to \F_2,\quad  a \mapsto 1 + \rot_\xi(\alpha) \bmod 2 
$$
to $\xi$, where $\alpha$ is an imbedded circle that represents $a$. Denote the Arf invariant of this form by $\Arf(\xi)$.

\begin{theorem}[Kawazumi]
\label{thm:kawazumi}
Suppose that $\xi_0$ and $\xi_1$ are framings of $S$.
\begin{enumerate}

\item If $g=0$, then $\xi_0$ and $\xi_1$ are in the same $\G_{0,m}$ orbit if and only if $\dd(\xi_0) = \dd(\xi_1)$.

\item If $g > 1$ and $\dd(\xi_0)$ is not even, then $\xi_0$ and $\xi_1$ are in the same $\G_{g,m}$-orbit if and only if $\dd(\xi_0) = \dd(\xi_1)$.


\item If $g > 1$ and $\dd(\xi_0)$ is even, then $\xi_0$ and $\xi_1$ are in the same $\G_{g,m}$-orbit if and only if $\dd(\xi_0) = \dd(\xi_1)$ and $\Arf(\xi_0) = \Arf(\xi_1)$.

\item If $g=1$, then $\xi_0$ and $\xi_1$ are in the same $\G_{1,m}$-orbit if and only if $\dd(\xi_0) = \dd(\xi_1)$ and $A(\xi_o) = A(\xi_1)$, where
$$
A(\xi) := \gcd\{\rot_\xi(\alpha) : \alpha \text{ is a non-separating simple closed curve in } S\}.
$$

\end{enumerate}

\end{theorem}

\begin{remark}
The role of the quadratic form $F_\xi$ is not mysterious. When $\dd(\xi)$ is even, there is a unique ``square root'' $\sqrt{\xi}$ of $\xi$. It is a section of a rank 2 vector bundle that is a square root of $T\Sbar$
whose local degree at $q\in Q$ is $d(\xi)/2$. This bundle corresponds to a spin structure on $\Sbar$. As is well known, spin structures correspond to $\F_2$ quadratic forms on $H_1(\Sbar;\F_2)$. There are only two $\Sp(H_\Z)$ orbits of these, and they are distinguished by the Arf invariant.
\end{remark}

We can regard the topological tangent bundle $T\Sbar$ of the oriented surface $\Sbar$ as complex line bundle $T$. This allows us to define the section $\xi_o^m$ of the complex line bundle $T^{\otimes m}$ over $S$ for all $m>0$. These are well defined up to homotopy. The obstruction to two ``even'' framings being in the same mapping class group orbit vanishes when we take squares.

\begin{corollary}
\label{cor:quad}
When $g>1$, $\xi_0^2$ and $\xi_1^2$ are in the same $\G_{g,m}$ orbit if and only if $\dd(\xi_0) = \dd(\xi_1)$.
\end{corollary}

\begin{proof}
If any $d_j$ is odd or if all $d_j$ are even and $\Arf(\xi_0)=\Arf(\xi_1)$, then Kawazumi's result implies that $\xi_0$ and $\xi_1$ are in the same mapping class group orbit, so $\xi_0^2$ and $\xi_1^2$ are as well. Now suppose that all $d_j$ are even and that $\Arf(\xi_0) \neq \Arf(\xi_1)$. Observe that $\xi_0^2$ has $2^{2g}$ square roots. These are indexed by elements $\delta$ of $H^1(\Sbar;\F_2)$. We need to give a construction of the corresponding square root $\xi_\delta$ of $\xi_0^2$.

Denote the flat line bundle of order 2 over $\Sbar$ that corresponds to $\delta \in H^1(\Sbar;\F_2)$ by $L_\delta$. Since $L_\delta$ is topologically trivial, $T\Sbar \otimes L_\delta \cong T\Sbar$. Since $L_\delta$ is flat, its square is canonically isomorphic to the trivial bundle, so that one has a canonical isomorphism
$(T\Sbar\otimes L_\delta)^{\otimes 2} \cong (T\Sbar)^{\otimes 2}$. Thus one has the commutative diagram
$$
\xymatrix{
T\Sbar\otimes L_\delta \ar[r]^{(\phantom{x})^2}\ar[d]  & (T\Sbar)^{\otimes 2} \ar[d] \cr
\Sbar \ar@/^1pc/[u]^{\pm \xi_\delta} \ar@{=}[r] & \Sbar \ar@/_1pc/[u]_{\xi_0^2}
}
$$
where the top row is the squaring map $v\mapsto v^{\otimes 2}$. The preimage of the image of $\xi_0^2$ under the squaring map splits into two components. The section corresponding to either is the square root $\xi_\delta$ of $\xi_0^2$. Observe that
$$
F_{\xi_\delta} = F_{\xi_0} + \delta
$$
Choose $\delta$ so that $\Arf(F_{\xi_0} + \delta) = \Arf(\xi_1)$. Then $\Arf(\xi_\delta) = \Arf(\xi_1)$, so that $\xi_\delta$ and $\xi_1$ are in the same orbit. Since $\xi_0^2 = \xi_\delta^2$, this implies that $\xi_0^2$ and $\xi_1^2$ are in the same orbit.
\end{proof}

\section{The Existence of Quasi-algebraic Framings}
\label{sec:alg_framing}

In this section we prove Theorem~\ref{thm:quasi-alg}. We first fix the notation to be used in this and subsequent sections.

Suppose that $2g+n>1$, where $g$ and $n$ are non-negative integers. Suppose that $S$ is an $(n+1)$-punctured surface of genus $g$. Write $S=\Sbar-P$, where $P=\{x_0,\dots,x_n\}$ is a subset of $\Sbar$. Fix a vector $\dd = (d_0,\dots,d_n) \in \Z^{n+1}$ with $\sum_0^n d_j = 2-2g$. Suppose that $\vv_o$ is a non-zero tangent vector of $\Sbar$ anchored at the point $x_0$ and that $\xi_o$ is a nowhere vanishing vector field on $S$ with local degree $d_j$ at $x_j$.

A complex structure on $(\Sbar,P)$ is an orientation preserving diffeomorphism
\begin{equation}
\label{eqn:c_str}
\phi : (\Sbar,P) \to (\Xbar,D),
\end{equation}
where $\Xbar$ is a compact Riemann surface and $D$ a finite subset. It induces the complex structure $(\Sbar,P,\vv_o) \to (\Xbar,D,\phi_\ast \vv_o)$ on $(\Sbar,P,\vv_o)$. A complex structure $\phi : (\Sbar,P) \to (\Xbar,D)$ determines a base point of $\M_{g,n+1}$ and a natural isomorphism $\phi_\ast : \G_{g,n+1} \to \pi_1(\M_{g,n+1},\phi)$.

\begin{definition}
Suppose that $m$ is a positive integer. A {\em complex structure on $(\Sbar,P,\xi_o^m)$} (or on $\xi_o^m$ for short) is a complex structure $\phi : (\Sbar,P) \to (\Xbar,D)$ on $(\Sbar,P)$ and a meromorphic section $\eta$ on $\Xbar$ of $(T\Xbar)^{\otimes m}$ (an {\em algebraic $m$-framing}) whose divisor is supported on $D$ and whose pullback $\phi|_S^\ast\eta$ to $S$ is homotopic to $\xi_o^m$. A {\em quasi-complex structure} on $(\Sbar,P,\xi_o)$ is a complex structure on $(\Sbar,P,\xi_o^m)$ for some $m>0$. These correspond to quasi-algebraic framings on $(\Xbar,D)$.
\end{definition}

\begin{remark}
The residue theorem implies that $(\Sbar,P,\xi_o)$ does not have a complex structure when, say, $d_0 = 1$ and all other $d_j < 0$ are negative. However, $\xi_o^2$ can have a complex structure in this case. For example, suppose that $g\ge 1$ and that $\Xbar$ is the smooth projective model of the hyperelliptic curve
$$
y^2 = \prod_{j=0}^{2g} (x-a_j),
$$
where the $a_j$ are distinct elements of $\C^\ast$. Let $x_j$ be the point of $\Xbar$ lying over $a_j$ and $x_{2g+1}$ the point lying over $\infty$. Then the meromorphic section $\eta$ of $(T\Xbar)^{\otimes 2}$ dual to the quadratic differential
$$
\omega := \prod_{j=0}^{2g}(x-a_j)^{-d_j}\bigg(\frac{dx}{y}\bigg)^2
$$
is a 2-framing. It has divisor $2\sum_{j=0}^{2g+1} d_jx_j$. Each square root of $\omega$ is a topological framing of $X$. In particular, we can take $d_0 = 1$ and all other $d_j \le 0$.
\end{remark}

\begin{remark}
When $g=1$ and $\dd = 0$, $X$ is a punctured elliptic curve. So $\Xbar = \C/\Lambda$ for some lattice $\Lambda$. Since $T\Xbar$ is a trivial holomorphic line bundle, the only holomorphic sections of $(T\Xbar)^{\otimes m}$ are multiplies of the translation invariant section $(\partial/\partial z)^m$. All other smooth sections $\xi$ of $TX$ with $\dd=0$ differ from it by an element $e(\xi)$ of $H^1(\Xbar;\Z)$. If $e(\xi)\neq 0$, then $(S,\xi)$ does not  admit a quasi-complex structure.
\end{remark}

\begin{proposition}
\label{prop:framings}
For each $g$ below, $\dd\in \Z^{n+1}$ satisfies $\sum_0^n d_j = 2-2g$.
\begin{enumerate}

\item\label{item:g=0} If $g=0$, then for all $\dd$, there is exactly one mapping class group orbit of homotopy classes of complex structure on $(\Sbar,P,\xi_o)$.

\item\label{item:holo} If $g > 3$ and $\dd$ satisfies $d_j < 0$ for $j=0,\dots,n$, then there is at least one mapping class group orbit of complex structures on $(\Sbar,P,\xi_o)$.

\item\label{item:quadratic} If $g>1$, then there is at least one mapping class group orbit of homotopy class of complex structure on $(\Sbar,P,\xi_o^2)$ for all $\dd$.

\item\label{item:adapted} If $g=1$ and $\dd=0$, then there is exactly one complex structure on $(\Sbar,P,\xi_o)$.

\item\label{item:g=1} If $g=1$ and $\dd\neq 0$, then there is a quasi-complex structure on $(\Sbar,P,\xi_o)$ if and only if $A(\xi_o) = \gcd\{d_0,\dots,d_n\}$.\footnote{The condition that $A(\xi_o) = \gcd\{d_0,\dots,d_n\}$ is equivalent to the condition that $\rot_{\xi_o}(\alpha)$ is divisible by $\gcd(d_j)$ for all simple closed curves $\alpha$ in $S$.}

\item\label{item:conv} If $g=1$, $\#\{j : d_j \neq 0\} > 2$ and $A(\xi_0) = \gcd\{d_0,\dots,d_n\}$, then $(\Sbar,P,\xi_0)$ has a complex structure for all complex structures $(\Xbar,D)$ on $(\Sbar,P)$.

\end{enumerate}
\end{proposition}

\begin{proof}
The proof of the genus 0 case (\ref{item:g=0}) is elementary and is left to the reader. We now assume that $g>0$.

Suppose now that $g>1$. Denote the locus of $(n+1)$-pointed curves $(C;x_0,\dots,x_n)$ in $\M_{g,n+1}$ for which $m\sum_j d_j x_j$ an $(-m)$-canonical divisor by $\sS_\dd^m$. This locus may be empty and may be disconnected. Each connected component of $\sS_\dd^m$ determines a $\G_{g,n+1}$-orbit of $m$-framings $\xi$ of the punctured reference surface $S$. When $g>3$, the classification of strata of abelian differentials \cite[Thm.~1]{kontsevich-zorich} implies that if all $d_j<0$, then $\sS_\dd^1$ is non-empty when at least one $d_j$ is odd and that it has exactly 2 non-hyperelliptic components, distinguished by the Arf invariant, when all $d_j$ are even. This and Theorem~\ref{thm:kawazumi} imply (\ref{item:holo}). The classification of meromorphic differentials in \cite{mero_diffs} implies that $\sS_\dd^2$ is non-empty all $\dd$. Combined with Corollary~\ref{cor:quad} it proves (\ref{item:quadratic}).

Suppose now that $g=1$ and that $\Xbar = \C/\Lambda$. Every algebraic $m$-framing of $(\Xbar,D)$ is of the form
$$
\eta = f(z) \big(\partial/\partial z)^m
$$
where $f$ is a non-zero meromorphic function whose divisor is $\sum_j d_j x_j$, where $D=\{x_0,\dots,x_n\}$. If $\xi$ is an $m$th root of $\eta$, then as $A(\partial/\partial z)=0$, it follows that the rotation number $\rot_\xi(\gamma)$ of every closed curve in $X$ lies in the ideal generated by the $d_j$. It follows that $A(\xi) = \gcd(d_j)$ for all quasi-algebraic framings of $X$. This proves (\ref{item:adapted}) and the ``only if'' part of (\ref{item:g=1}). If $\dd = \pm(-m,m)$, where $m>0$, then we can take $x_1-x_0$ to be a non-zero $m$-torsion point of the jacobian of $X$ and $f$ to be a function whose divisor is $m(x_1-x_0)$. We prove the remainder of the converse by proving (\ref{item:conv}).

Suppose that $g=1$ and $\dd\neq 0$. By decreasing $n$ if necessary, we may assume that all $d_j$ are non-zero. Suppose that $n>1$. Define
$$
F_\dd : X^{n+1} \to \Jac X
$$
by $F_\dd(x_0,\dots,x_n) = \sum_j d_j x_j$. We have to show that the fiber $Y$ of $F_\dd$ over $0$ is not contained in any of the diagonals $\Delta_{j,k} := \{x_j = x_k\}$. To see that $Y$ cannot be contained in $\Delta_{j,k}$, choose $\ell$ such that $j,k,\ell$ are distinct. This is possible as $n>1$. If $(x_0,\dots,x_n)\in Y$ then for all but finitely many $u \in \Jac X$, $(y_0,\dots,y_n)$ is not in $\Delta_{j,k}$, where
$$
y_a := 
\begin{cases}
x_a & a \neq k,\ell,\cr
x_k + d_\ell u & a = k, \cr
x_\ell - d_k u & a = \ell.
\end{cases}
$$
This completes the proof of (\ref{item:g=1}) and (\ref{item:conv}).
\end{proof}

\begin{remark}
\label{rem:framings}
This result implies that the framings that occur in \cite[Thm.~6.1]{akkn2} are precisely those that admit a quasi-complex structure. See footnote~\ref{foot:rot} on page~\pageref{foot:rot} for conventions.
\end{remark}

\section{Torsors of Splittings of the Goldman--Turaev Lie Bialgebra}
\label{sec:hoge_splittings}

In this section, we explain how Hodge theory gives torsors of simultaneous splittings of (\ref{eqn:split_GT}) and (\ref{eqn:split_A}) and explain how these give solutions to the Kashiwara--Vergne problem. In particular, we prove Corollary~\ref{cor:big_kv} and take the first steps towards proving Theorem~\ref{thm:krv}.

\begin{proposition}
\label{prop:splittings}
Each homotopy class of quasi-complex structures on $(\Sbar,P,\vv_o,\xi_o)$ gives a torsor of simultaneous splittings of (\ref{eqn:split_GT}) and (\ref{eqn:split_A}). The splittings constructed from a fixed complex structure on $(\Sbar,P,\vv_o) \to (\Xbar,D,\vv_o)$ are torsors under the prounipotent radical $\U^\MT_{X,\vv}$ of the Mumford--Tate group of $\Q\pi_1(X,\vv)^\wedge$.
\end{proposition}

\begin{proof}
By \cite[Thm.~6]{hain:goldman}, the MHS on $\Q\pi_1(X,\vv)^\wedge$ determines a torsor of isomorphisms
\begin{equation}
\label{eqn:KVI}
\Q\pi_1(X,\vv)^\wedge \To \prod_{m\le 0} \Gr^W_m \Q\pi_1(X,\vv)^\wedge
\end{equation}
each of which solves the KV-problem $\KVI^{(g,n+1)}$, as defined in \cite{akkn1}. These are a torsor under $\U^\MT_{X,\vv}$. Corollary~\ref{cor:gr_gt} implies (via the discussion in \cite[\S10.2]{hain:goldman}) that the induced isomorphism
$$
\Q\lambda(X)^\wedge \cong \prod_{m\le 0} \Gr^W_m \Q\lambda(X)^\wedge
$$
is an isomorphism of Lie bialgebras.
\end{proof}

These Hodge theoretic splittings give solutions to the KV-problem $\KV^{(g,n+1)}_\dd$. This result implies Corollary~\ref{cor:big_kv}.

\begin{corollary}
\label{cor:kv}
Each homotopy class of quasi-complex structures on $(\Sbar,P,\vv_o,\xi_o)$ gives a torsor of solutions to the Kashiwara--Vergne problem $\KV^{(g,n+1)}_\dd$. These solutions form a torsor under the prounipotent radical $\U^\MT_{X,\vv}$ of the Mumford--Tate group of $\Q\pi_1(X,\vv)^\wedge$.
\end{corollary}

\begin{proof}
This follows from Proposition~\ref{prop:splittings} and \cite[Thm.~5]{akkn1}, which implies that the automorphism $\Phi$ of
$$
\Q\ll x_1,\dots,x_g,y_1,\dots,y_g,z_1,\dots,z_n\rr
$$
constructed from the choice of a lifting $\chitilde$ of the canonical central cocharacter $\chi: \Gm \to \pi_1(\MHS^\ss)$ in \cite[\S13.4]{hain:goldman} is a solution of $\KVI^{(g,n+1)}$.
\end{proof}

\begin{remark}
In view of Remark~\ref{rem:framings}, this gives a new and independent proof of the main result, Theorem~6.1, of \cite{akkn2}.
\end{remark}

Solutions of $\KV^{(g,n+1)}_\dd$ that arise from Hodge theory will be called {\em motivic solutions} as they arise from a complex (and thus algebraic) structure on $(\Sbar,P,\vv_o,\xi_o^m)$ for some $m>0$. All solutions of $\KV^{(g,n+1)}_\dd$ comprise a torsor under a prounipotent subgroup $\KRV_{g,n+\uu}^\dd$ of $\Aut\Q\pi_1(S,\vv_o)^\wedge$. For each complex structure $\phi$ on $(\Sbar,P,\vv_o,\xi_o^m)$, there is an inclusion $\phi_\ast : \U_{X,\vv}^\MT \hookrightarrow \KRV_{g,n+\uu}^\dd$. These homomorphisms depend non-trivially on $\phi$ and are, in general, not surjective.

\section{The Stabilizer of a Framing}
\label{sec:stabilizer}

A second way to generate solutions of the KV-problem $\KV^{(g,n+\uu)}_\dd$ from a given solution is to conjugate it by an element of the Torelli group $T_{g,n+\uu}$ (defined below) that fixes the framing $\xi_o$. In this section, we compute the stabilizer of a framing.

Suppose that $\Sbar$ is a compact oriented surface of genus $g$ and that $2g-2+m+r>0$. For each commutative ring $A$ set $H_A = H_1(\Sbar;A)$. The intersection pairing $H_A\otimes H_A \to A$ is a unimodular symplectic form. Denote the corresponding symplectic group by $\Sp(H_A)$. We will regard both $H$ and $\Sp(H)$ as affine groups over $\Z$ whose $A$-rational points are $H_A$ and $\Sp(H_A)$, respectively. The Torelli group $T_{g,m+\vec{r}}$ is defined to be the kernel of the homomorphism
$$
\rho : \G_{g,m+\vec{r}} \to \Sp(H_\Z)
$$
that is induced by the action of $\G_{g,m+\vec{r}}$ on $H_\Z$. This homomorphism is well-known to be surjective.

For the remainder of this section $(\Sbar,P)$ will be an $(n+1)$ pointed surface of genus $g$, where $2g-2+n>0$, and $\xi_o$ will be a framing of $S$ with vector of local degrees $\dd$. Denote the push forward of $\xi_o$ by $\psi \in \Diff^+(\Sbar,P)$ by $\psi_\ast \xi_o$. The homotopy class of this push forward depends only on the class of $\psi$ in the mapping class group $\G_{g,n+1}$ of $(\Sbar,P)$. Since $\psi$ fixes the punctures $P$, $\psi_\ast \xi_o$ and $\xi_o$ have the same local degrees. The homotopy class of their ratio
$$
(\psi_\ast \xi_o)/\xi_o : \Sbar \to \C^\ast
$$
is an element of $H^1(\Sbar;\Z)$ that we denote by $f_{\xi_o}(\psi)$. It vanishes if and only if $\psi$ fixes $\xi_o$.

\begin{lemma}
The function $f_{\xi_o} : \Gamma_{g,n+1} \to H_\Z$ is a 1-cocycle. Its restriction to the Torelli group $T_{g,n+1}$ is an $\Sp(H_\Z)$-equivariant homomorphism whose kernel is the stabilizer of $\xi_o$ in $T_{g,n+1}$.
\end{lemma}

\begin{proof}
It is clear from the definition that $\psi\in \G_{g,n+1}$ stabilizes the homotopy class of $\xi_o$ if and only if $f_{\xi_o}(\psi) = 0$. Suppose that $\psi',\psi'' \in \G_{g,n+1}$. Since
$$
\frac{(\psi'\psi'')_\ast \xi_o}{\xi_o} = \frac{\psi'_\ast\xi_o}{\xi_o} \frac{(\psi'\psi'')_\ast\xi_o}{\psi'_\ast\xi_o}
= \frac{\psi'_\ast\xi_o}{\xi_o}\cdot \psi'_\ast\bigg(\frac{\psi''_\ast \xi_o}{\xi_o}\bigg)
$$
as homotopy classes of functions $\Sbar \to \C^\ast$, it follows that $f_{\xi_o}$ satisfies the 1-cocycle condition
$$
f_{\xi_o}(\psi'\psi'') = f_{\xi_o}(\psi') + \psi'_\ast f_{\xi_o}(\psi'').
$$
The restriction of $f_{\xi_o}$ to $T_{g,n+1}$ is a homomorphism as the Torelli group acts trivially on $H_\Z$.
\end{proof}

In the next section, we will need to know that the class of $f_{\xi_o}$ is a Hodge class. In preparation for proving this, we give an algebro-geometric interpretation of $f_{\xi_o}$.

The vector of local degrees $\dd$ of $\xi_o$ determines a section $F_\dd$ of the universal jacobian $\cJ_{g,n+1}$ over $\M_{g,n+1}$. It is defined by
\begin{equation}
\label{eqn:normal}
F_\dd(C;x_0,\dots,x_n) = K_C + \sum_{j=0}^n d_j x_j \in \Jac C
\end{equation}
where $C$ is a compact Riemann surface of genus $g$; $x_0,\dots,x_n$ are distinct labelled points of $C$; and where $K_C$ denotes the canonical class of $C$.\footnote{The image of $(C;x_0,\dots,x_n)$ under $F_\dd$ corresponds to a $C^\infty$ isomorphism of the line bundle $\cO_C(\sum d_j x_j)$ with $TC$ under which the section of $\cO_C(\sum d_j x_j)$ with divisor $\sum d_j x_j$ corresponds to a framing with local degree vector $\dd$. This gives an $m$-framing of $C-\{x_0,\dots,x_n\}$ if and only if $F_\dd(C;x_0,\dots,x_n)$ is an $m$-torsion point of $\Jac C$. If $g\neq 1$, or if $g=1$ and $\gcd\{d_j\} = A(\xi_o)$, this gives a complex structure on $\xi_o^m$.}

Fix a base point $o$ of $\M_{g,n+1}$. Denote the identity of $\Jac C_o$ by $z_o$. The fundamental group of $\cJ_{g,n+1}$ with base point $z_o$ is an extension of $\Gamma_{g,n+1}$ by $H_\Z$. The identity section induces a splitting of this extension and thus a canonical isomorphism
$$
\pi_1(\cJ_{g,n+1},z_o) \cong \G_{g,n+1}\ltimes H_\Z
$$
where we are identifying $\pi_1(\M_{g,n+1},o)$ with $\Gamma_{g,n+1}$ and $H_\Z$ with $H_1(C_o;\Z)$. The standard representation $\G_{g,n+1} \to \Sp(H_\Z)$ induces a homomorphism
$$
\pi_1(\cJ_{g,n+1},z_o) \to \Sp(H_\Z)\ltimes H_\Z.
$$
The section $F_\dd$ of $\cJ_{g,n+1}$ over $\M_{g,n+1}$ induces a homomorphism
$$
\tau_\dd : \Gamma_{g,n+1} \to \pi_1(\cJ_{g,n+1},z_o) \to \Sp(H_\Z) \ltimes H_\Z,
$$
which determines a cohomology class $[\tau_\dd] \in H^1(\Gamma_{g,n+1},H_\Z)$.

\begin{proposition}
The cohomology classes of $f_{\xi_o}$ and $\tau_\dd$ in $H^1(\G_{g,n+1};H_\Z)$ are equal. In particular, the class of $f_{\xi_o}$ depends only on the vector $\dd$ of local degrees.
\end{proposition}

\begin{proof}[Sketch of Proof]
These classes clearly vanish when $g=0$. So suppose that $g>0$. First observe that $H^1(\G_{g,n+1};H_\Z)$ is torsion free. This can be proved using the cohomology long the exact sequence of
$$
\xymatrix{
0 \ar[r] & H_\Z \ar[r]^{\times N} & H_\Z \ar[r] & H_{\Z/N} \ar[r] & 0
}
$$
the vanishing of $H^0(\G_{g,n+1};H_{\Z/N})$ for all $N>0$, and the finite generation of $H^j(\G_{g,n+1};H_\Z)$. It therefore suffices to show that the classes of $f_{\xi_o}$ and $\tau_\dd$ agree in $H^1(\Gamma_{g,n+1};H_\Q)$.

By the ``center kills'' argument $H^\bdot(\Sp(H_\Z);H_\Q)$ vanishes. This implies (via the Hochschild--Serre spectral sequence) that the restriction mapping
$$
H^1(\G_{g,n+1};H_\Q) \to \Hom_{\Sp(H_\Z)}(H_1(T_{g,n+1}),H_\Q)
$$
is an isomorphism.

Denote the pure braid group on $n+1$ strings of $\Sbar$ by $\pi_{g,n+1}$. The inclusion of the configuration space of $\Sbar$ into $\Sbar^{n+1}$ induces an isomorphism $H_1(\pi_{g,n+1}) \cong H_\Z^{n+1}$. (See \cite[Prop.~2.1]{hain:torelli}.) When $g>1$, the inclusion $\pi_{g,n+1} \to T_{g,n+1}$ induces an isomorphism
\begin{equation}
\label{eqn:pt_push}
\Hom_{\Sp(H_\Z)}(H_1(T_{g,n+1}),H_\Q) \to \Hom_{\Sp(H_\Z)}(H_1(\pi_{g,n+1}),H_\Q) \cong \Q^{n+1}.
\end{equation}
When $g>2$, this follows from Johnson's work \cite{johnson:h1} as in \cite[Prop.~4.6]{hain:torelli}. When $g=2$, this follows similarly from results in \cite{watanabe}. When $g=1$, there is an exact sequence
$$
0 \to H_\Z \to H_1(\pi_{g,n+1}) \to H_1(T_{1,n+1}) \to 0
$$
where the left-hand map is the diagonal embedding, which is induced by the diagonal action of an elliptic curve $E$ on $E^{n+1}$. This implies that (\ref{eqn:pt_push}) is injective with image the hyperplane consisting of those $(u_0,\dots,u_n)$ with $\sum_j u_j = 0$.

These observations imply that to prove the equality of the classes of $f_{\xi_o}$ and $\tau_\dd$, we just have to see that they agree on the ``point pushing'' subgroup of $T_{g,n+1}$; that is, on the image of $\pi_{g,n+1}$ in the Torelli group.

Both $f_{\xi_o}$ and $\tau_\dd$ have image $\dd \in \Q^{n+1}$. The computations for $\tau_\dd$ can be found in \cite[Prop.~11.2]{hain:normal} for $g>1$ and \cite[\S12]{hain:normal} for $g=1$. We now sketch a proof for $f_{\xi_o}$.\footnote{Morita \cite[Prop.~4.1]{morita:casson} has proved the $n=0$ case.}

Suppose that $\alpha$ is a loop in $S$ whose closure is a loop in $\Sbar$ based at $x_j$. Denote the corresponding ``point pushing element of $\G_{g,n+1}$, that ``pushes $x_j$ around $\alpha$'', by $\psi_\alpha$. For all loops $\gamma$ in $S$ we have
$$
f_{\xi}(\psi_\alpha) : \gamma \mapsto \rot_{\psi_\ast\xi}(\gamma)- \rot_\xi(\gamma).
$$
We have to show that $f_{\xi}(\psi_\alpha) = d_j(\alpha\cdot \blank)$, where $(\blank\cdot\blank)$ denotes the intersection pairing. Equivalently, we have to show that
$$
\rot_{\psi_{\alpha\ast\xi}}(\gamma)- \rot_\xi(\gamma) = d_j (\alpha\cdot\gamma).
$$
It suffices to prove this when $\alpha$ is a simple closed curve and when both the geometric and algebraic intersection numbers of $\alpha$ and $\gamma$ are 1.

Since $\alpha$ is a simple closed curve, it has a regular neighbourhood $A \approx S^1 \times [-1,1]$ that is an annulus where $\alpha$ corresponds to $S^1\times 0$. We may assume that $\gamma$ intersects $A$ in the interval $1\times [-1,1]$ and that $\psi_\alpha$ is supported in $A$. We refer to Figure~\ref{fig:point_push} for additional notation.
\begin{figure}[!ht]
\epsfig{file=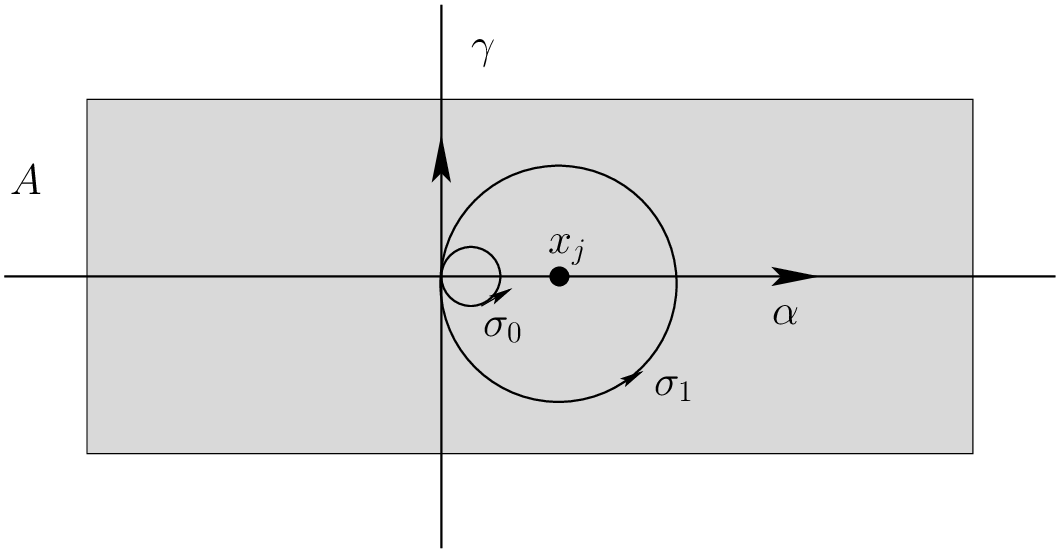, width=4in}
\caption{Point pushing}
\label{fig:point_push}
\end{figure}
Regard $\gamma$ as a loop based at the point of intersection of $\gamma$ and $\alpha$. It is homotopic to $\gamma\sigma_0$ in $S$ and $\psi_\alpha(\gamma)$ is homotopic to $\gamma\sigma_1$. Applying the formula in the footnote on page~\pageref{page:foot}, we have
\begin{multline*}
f_{\xi}(\psi_\alpha^{-1})(\gamma) 
= \rot_{\psi_{\alpha\ast}^{-1}\xi}(\gamma) - \rot_\xi(\gamma)
= \rot_\xi(\psi_\alpha\circ\gamma) - \rot_\xi(\gamma)\cr
= \rot_\xi(\sigma_1) - \rot_\xi(\sigma_0)
= (1-d_j) - 1
= -d_j.
\end{multline*}
This implies that $f_{\xi}(\psi_\alpha) = - f_{\xi}(\psi_\alpha^{-1}) = d_j = d_j(\alpha,\gamma)$.
\end{proof}

\begin{corollary}
The stabilizer of $\xi_o$ in $\G_{g,n+1}$ equals the kernel of $\tau_\dd : \G_{g,n+1} \to \Sp(H_\Z)\ltimes H_\Z$.
\end{corollary}

\begin{proof}
This follows from the general fact that the kernel of an element of a cohomology group $H^1(\G,V)$ is well defined. This is because elements of this group correspond to lifts of the action $\G \to \Aut(V)$ of $\G$ on $V$ to a homomorphism $\G \to \Aut(V)\ltimes V$ modulo conjugation by elements of $V$. Conjugating such a homomorphism by an element of $V$ does not change its kernel.
\end{proof}

\section{Relative Completion of Mapping Class Groups and Torsors of Splittings}
\label{sec:torsors}

In this section, we consider the torsor of splittings of the Goldman--Turaev Lie bialgebra obtained by combining those constructed in Section~\ref{sec:hoge_splittings} using Hodge theory with those coming from the stabilizer of $\xi_o$ in the Torelli group. We will use the notation of the previous section. We replace mapping class groups by their relative completions, which allows us to prove stronger results.

Recall from \cite{hain:torelli} that the completion of $\G_{g,m+\vec{r}}$ relative to $\rho : \G_{g,m+\vec{r}} \to \Sp(H_\Q)$ is an affine $\Q$-group $\cG_{g,m+\vec{r}}$ that is an extension
$$
1 \to \U_{g,m+\vec{r}} \to \cG_{g,m+\vec{r}} \to \Sp(H) \to 1
$$
of affine $\Q$-groups, where $\U_{g,m+\vec{r}}$ is prounipotent. There is a Zariski dense homomorphism $\rhotilde : \G_{g,m+\vec{r}} \to \cG_{g,m+\vec{r}}(\Q)$ whose composition with the homomorphism $\cG_{g,m+\vec{r}}(\Q) \to \Sp(H_\Q)$ is $\rho$. When $g=0$, $\Sp(H)$ is trivial and $\cG_{0,m+\vec{r}}$ is the unipotent completion $\G_{0,m+\vec{r}}^\un$.

\begin{remark}
\label{rem:density}
The homomorphism $T_{g,m+\vec{r}} \to \U_{g,m+\vec{r}}(\Q)$ induced by $\rhotilde$ has Zariski dense image when $g>1$. This follows from the right exactness of relative completion \cite[Thm.~3.11]{hain:morita} and the vanishing of $H^1(\Sp(H_\Z);V)$ for all rational representations $V$ of $\Sp(H)$ when $g\neq 1$.  (See \cite[Thm.~4.3]{hain:morita}.) However, when $g=1$, $T_{1,n+\uu} \to \U_{1,n+\uu}(\Q)$ is not Zariski dense. For example, $T_{1,1}$ is trivial, while the Lie algebra of $\U_{1,1}$ is freely topologically generated by an infinite dimensional vector space as explained in \cite[Remarks~3.9,7.2]{hain:torelli} and in \cite[\S10]{hain-matsumoto:mem}.
\end{remark}

The action of the mapping class group $\G_{g,n+\uu}$ on $\Q\pi_1(S,\vv_o)$ induces an action on $\Q\lambda(S)$ which preserves the Goldman bracket. The stabilizer of $\xi_o$ preserves the Turaev cobracket. The universal mapping property of relative completion implies that $\cG_{g,n+\uu}$ acts on $\Q\pi_1(S,\vv_o)^\wedge$ and $\Q\lambda(S)^\wedge$. Since the image of the mapping class group in $\cG_{g,n+\uu}$ is Zariski dense, this action preserves the Goldman bracket. However, since $\cG_{g,n+\uu}$ does not generally preserve framings, it is not clear which subgroup of $\U_{g,n+\uu}$ preserves the cobracket. Our next task is to determine this subgroup.

The universal property of relative completion implies that the homomorphism $\tau_\dd : \G_{g,n+\uu} \to \Sp(H_\Z)\ltimes H_\Z$ constructed Section~\ref{sec:stabilizer} induces a homomorphism
$$
\tautilde_\dd : \cG_{g,n+\uu} \to \Sp(H)\ltimes H.
$$
It is surjective as the image of $\tau_\dd$ is Zariski dense in $\Sp(H)\ltimes H$.

\begin{proposition}
For all quasi-algebraic framings $\xi_o$ of $S$, the action of $\ker\tautilde_\dd$ on $\Q\lambda(S)^\wedge$ preserves the completed Turaev cobracket
\begin{equation}
\label{eqn:cobracket}
\delta_{\xi_o} : \Q\lambda(S)^\wedge \to \Q\lambda(S)^\wedge \comptensor \Q\lambda(S)^\wedge.
\end{equation}
\end{proposition}

\begin{proof}
When $g=0$, $\tautilde_\dd$ is trivial. Since $\G_{0,n+\uu}$ preserves the homotopy class of $\xi_o$, the result is trivially true. Now assume that $g>0$. For the rest of the proof, we assume the reader is familiar with the general theory of relative completion as explained in \cite[\S3]{hain:morita}.

When $g\ge 2$, every framing is quasi-algebraic by Proposition~\ref{prop:framings} and the algebraic nature of the framing will not play any explicit role in the proof. The computation \cite[Ex.~3.12]{hain:morita} and the right exactness of relative completion \cite[Prop.~3.7]{hain:morita} imply that the completion of $\Sp(H_\Z)\ltimes H_\Z$ relative to the obvious homomorphism to $\Sp(H_\Q)$ is $\Sp(H)\ltimes H$; the canonical homomorphism $\Sp(H_\Z)\ltimes H_\Z \to \Sp(H_\Q)\ltimes H_\Q$ is the inclusion. Right exactness of relative completion implies that the sequence
$$
\xymatrix{
(\ker \tau_\dd)^\un \ar[r] & \cG_{g,n+\uu} \ar[r]^(.42){\tautilde_\dd} & \Sp(H) \ltimes H \ar[r] & 1
}
$$
is exact, where $(\phantom{X})^\un$ denotes unipotent completion. Since every group is Zariski dense in its unipotent completion, the exactness of this sequence implies that $\ker \tau_\dd$ is Zariski dense in $\ker \tautilde_\dd$. Since $\ker\tau_\dd$ fixes $\xi_o$, it preserves the completed cobracket. It follows that $\ker\tautilde_\dd$ does as well.

In view of Remark~\ref{rem:density}, the proof is more intricate when $g=1$. We first consider the case when $n=0$. We take $\Sbar$ to be the group $S^1\times S^1$ and $P$ to be its identity. In this case, $\xi_o$ is a translation invariant vector field. Since any two translation invariant vector fields are homotopic, it follows that their homotopy classes lie in one $\SL_2(\Z)$-orbit of framings. Since the cobracket depends only on the homotopy class of the framing, $\SL_2(\Z)$ preserves the completed cobracket (\ref{eqn:cobracket}). Since $\SL_2(\Z)$ is Zariski dense in $\cG_{1,\uu}$, it follows that it also preserves the cobracket. Since the image of $\tautilde_\dd : \cG_{1,\uu} \to \Sp(H)\ltimes H$ is $\Sp(H)$, it follows that $\ker\tautilde_\dd = \U_{1,\uu}$ preserves the completed cobracket.

Suppose now that $g=1$ and $n>0$. Since $\xi_o$ is quasi-algebraic, $A(\xi_o)=\gcd(d_j)$. So there exist two transversely intersecting simple closed curves $\alpha$ and $\beta$ in $S$ with $\rot_{\xi_o}(\alpha) = \rot_{\xi_o}(\beta) = 0$. A regular neighbourhood of the union of $\alpha \cup \beta$ is a genus 1 surface with one boundary component. Since $\rot_{\xi_o}(\alpha)=\rot_{\xi_o}(\beta) = 0$, the restriction of the framing to the genus 1 subsurface is homotopic to a translation invariant framing.

The inclusion of the genus 1 subsurface induces an inclusion $\G_{1,\uu} \to \G_{1,n+\uu}$. By the $n=0$ case, the image of $\G_{1,\uu}$ in $\G_{1,n+\uu}$ preserves the homotopy class of $\xi_o$. The kernel of the restriction of $\tau_\dd : T_{1,n+\uu} \to H_\Z$ also preserves the class of $\xi_o$. So the subgroup
$$
\G_{\xi_o} := \langle \ker\tau_\dd \cap T_{1,n+\uu},\G_{1,\uu}\rangle
$$
of $\G_{1,n+\uu}$ generated by these two groups stabilizes the class of $\xi_o$ and thus preserves the cobracket. The prounipotent radical $\U_{\xi_o}$ of the Zariski closure of $\G_{\xi_o}$ in $\cG_{1,n+\uu}$ is generated by the image of $\U_{1,\uu}$ and the kernel of $\tau_\dd : T_{1,n+\uu}^\un \to H$. It is precisely the kernel of $\tau_\dd : \U_{1,n+\uu} \to H$. Since $\G_{\xi_o}$ preserves the cobracket, so does $\U_{\xi_o}$.
\end{proof}

Denote $\ker \tautilde_\dd$ by $\U_{g,n+\uu}^\dd$. There is a natural homomorphism
$$
\U_{g,n+\uu}^\dd \to \KRV^\dd_{g,n+\uu}.
$$
Denote the image of $\U_{g,n+\uu}$ in $\Aut \Q\pi_1(S,\vv_o)^\wedge$ by $\Ubar_{g,n+\uu}$ and the image of $\U_{g,n+\uu}^\dd$ by $\Ubar_{g,n+\uu}^\dd$.

A complex structure $\phi: (\Sbar,P,\vv_o) \to (\Xbar,D,\vv)$ determines a Mumford--Tate group $\MT_{X,\vv}$, which acts faithfully on $\Q\pi_1(X,\vv)^\wedge$. Denote the corresponding subgroup $\phi \MT_{X,\vv} \phi^{-1}$ of $\Aut\Q\pi_1(S,\vv_o)^\wedge$ by $\MT(\phi)$ and its prounipotent radical by $\U^\MT(\phi)$.

\begin{definition}
The group $\Uhat_{g,n+\uu}^\dd(\phi)$ is the subgroup of $\Aut \Q\pi_1(S,\vv_o)^\wedge$ generated by $\U^\MT(\phi)$ and $\Ubar_{g,n+\uu}^\dd$.
\end{definition}

Recall that a MHS on an affine $\Q$-group $G$ is, by definition, a MHS on its coordinate ring $\cO(G)$. Equivalently, a MHS on $G$ is an algebraic action of $\pi_1(\MHS)$ on $G$. A homomorphism $G_1 \to G_2$ of affine $\Q$-groups with MHS is a morphism of MHS if it is $\pi_1(\MHS)$ equivariant. A MHS on $G$ induces one on its Lie algebra.

\begin{lemma}
A quasi-complex structure $\phi$ on $(\Sbar,P,\vv_o,\xi_o)$ determines pro-MHS on the Lie algebras (and coordinate rings) of $\U_{g,n+\uu}^\dd$ and $\Uhat_{g,n+\uu}^\dd(\phi)$. The natural homomorphism $\Uhat_{g,n+\uu}^\dd(\phi) \to \Aut \Q\pi_1(X,\vv)^\wedge$ is a morphism of MHS.
\end{lemma}

\begin{proof}
The quasi-complex structure $\phi$ determines a MHS on $\U_{g,n+\uu}$. Observe that $\U_{g,n+\uu}^\dd$ is the kernel of the homomorphism
$$
\xymatrix{
\cG_{g,n+\uu} \ar[r] & \cG_{g,n+1} \ar[r]^(.4){\tautilde_\dd} & \Sp(H)\ltimes H
}
$$
induced on completed fundamental groups by the morphism of pointed varieties
$$
\xymatrix{
\big(\M_{g,n+\uu},(\Xbar,D,\vv)\big) \ar[r] & \big(\M_{g,n+1},(\Xbar,D)\big) \ar[r]^(.55){F_\dd} & \big(\X,(\Jac\Xbar,0)\big),
}
$$
where $\X \to \cA_g$ is the universal abelian variety over $\cA_g$, the moduli space of principally polarized abelian varieties. Since morphisms of pointed varieties induce morphisms of MHS on completed fundamental groups, it follows that $\U_{g,n+\uu}^\dd$ has a natural MHS.

This MHS corresponds to an action of $\pi_1(\MHS)$ on it, so that one has the group $\pi_1(\MHS)\ltimes \U_{g,n+\uu}^\dd$. The pro-MHS on $\Q\pi_1(X,\vv)^\wedge$ corresponds to a homomorphism $\pi_1(\MHS) \to \Aut \Q\pi_1(X,\vv)^\wedge$. By \cite[Lem.~4.5]{hain:torelli}, the homomorphism $\U_{g,n+\uu}^\dd \to \Aut \Q\pi_1(X,\vv)^\wedge$ is a morphism of MHS. It thus extends to a homomorphism
$$
\pi_1(\MHS)\ltimes \U_{g,n+\uu}^\dd \to \Aut \Q\pi_1(X,\vv)^\wedge.
$$
Its image is $\Uhat_{g,n+\uu}^\dd(\phi)$. The inner action of $\pi_1(\MHS)$ on $\Uhat_{g,n+\uu}^\dd$ gives it a MHS. The inclusion $\Uhat_{g,n+\uu}^\dd \hookrightarrow \Aut \Q\pi_1(X,\vv)^\wedge$ is $\pi_1(\MHS)$-invariant, which implies that it is a morphism of MHS.
\end{proof}

The following theorem is proved in Section~\ref{sec:krv}. It and the previous lemma imply Theorem~\ref{thm:krv}.

\begin{theorem}
\label{thm:krv2}
For each quasi-complex structure $\phi : (\Sbar,P,\vv_o,\xi_o) \to (\Xbar,D,\vv,\xi)$, there is an injective homomorphism $\Uhat_{g,n+\uu}^\dd(\phi) \hookrightarrow \KRV_{g,n+\uu}^\dd$ of prounipotent $\Q$-groups. Its image does not depend on the quasi-complex structure $\phi$. The group $\Ubar_{g,n+\uu}^\dd$ is a normal subgroup of $\Uhat_{g,n+\uu}^\dd$. There is a canonical surjection $\K \to \Uhat_{g,n+\uu}^\dd/\Ubar_{g,n+\uu}^\dd$, where $\K$ is the prounipotent radical of $\pi_1(\MTM)$.
\end{theorem}

Since $\Uhat_{g,n+\uu}^\dd(\phi)$ is independent of the choice of $\phi$, we denote it by $\Uhat_{g,n+\uu}^\dd$.

\begin{remark}
The complex structure on $(\Sbar,P,\vv_o,\xi_o)$ defines a $\C$-point, and thus a geometric point, $p$ of the moduli stack $\M_{g,n+\uu/\Q}$. Its \'etale fundamental group $\pi_1^\et(\M_{g,n+\uu},p)$ is an extension
$$
1 \to \G_{g,n+\uu}^\wedge \to \pi_1^\et(\M_{g,n+\uu},p) \to \Gal(\Qbar/\Q) \to 1.
$$
where $\G_{g,n+\uu}^\wedge$ denotes the profinite completion of the mapping class group. For each prime number $\ell$, there is an homomorphism $\pi_1^\et(\M_{g,n+\uu},p) \to \Sp(H_\Zl)\ltimes H_\Zl$. Denote its kernel by $\pi_1^\et(\M_{g,n+\uu},p)^\dd$. There is a homomorphism
$$
\phi_\ell : \pi_1^\et(\M_{g,n+\uu},p)^\dd \to \KRV_{g,n+\uu}^\dd(\Ql).
$$
Using weighted completion \cite[\S8]{hain:rat_pts}, one can show that the Zariski closure of the image of $\phi_\ell$ is $\Uhat_{g,n+\uu}^\dd(\Ql)$.
\end{remark}

Recall from \cite[\S10.2]{hain:goldman} that natural splittings of the weight filtration of a MHS correspond to lifts of the central cocharacter $\chi : \Gm \to \pi_1(\MHS^\ss)$ to $\pi_1(\MHS)$. Each MHS on the completed Goldman--Turaev Lie bialgebra and each lift of $\chi$ gives rise to a splitting of the Goldman--Turaev Lie bialgebra.\footnote{This is called {\em Goldman--Turaev formality} in \cite{akkn2}.} It also gives a grading of $\uhat_{g,n+\uu}^\dd$. Thus 

\begin{corollary}
Each choice of a quasi-complex structure $\phi : (\Sbar,P,\vv_o,\xi_o) \to (\Xbar,D,\vv,\xi)$ and each choice of a lift of the central cocharacter $\chi : \Gm \to \pi_1(\MHS^\ss)$ gives isomorphisms
$$
\u_{g,n+\uu}^\dd \cong \prod_m \Gr^W_m \u_{g,n+\uu}^\dd
\text{ and }
\Q\pi_1(X,\vv)^\wedge \cong \prod_m \Gr^W_m \Q\pi_1(S,\vv)
$$
such that the diagram
$$
\xymatrix{
\uhat_{g,n+\uu}^\dd \ar[r]\ar[d]^\cong & \Der \Q\pi_1(X,\vv)^\wedge\ar[d]^\cong \cr
\prod_m \Gr^W_m \uhat_{g,n+\uu}^\dd \ar[r] & \Der \prod_m \Gr^W_m \Q\pi_1(S,\vv)
}
$$
commutes. Each of these splittings descends to splitting of the weight filtration of the Goldman--Turaev Lie bialgebra $(\Q\lambda(S)^\wedge,\gold,\delta_{\xi_o})$. \qed
\end{corollary}

\section{Proof of Theorem~\ref{thm:krv2}}
\label{sec:krv}

We will use the notation of the previous section. We begin a reformulation of the definition of $\Uhat_{g,n+\uu}^\dd(\phi)$ associated to a quasi-complex structure
$$
\phi : (\Sbar,P,\vv_o,\xi_o)\to (\Xbar,D,\vv,\xi)
$$
on $(\Sbar,P,\vv_o,\xi_o)$. This determines an isomorphism $\G_{g,n+\uu} \cong \pi_1(\M_{g,n+\uu},\phi_o)$. The corresponding MHS on the relative completion $\cG_{g,n+\uu}$ corresponds to an action of $\pi_1(\MHS)$ on $\cG_{g,n+\uu}$. The quasi-complex structure $\phi$ determines a semi-direct product
$$
\pi_1(\MHS) \ltimes \cG_{g,n+\uu}.
$$
Since the natural homomorphism $\cG_{g,n+\uu} \to \Aut \Q\pi_1(X,\vv_o)^\wedge$ is a morphism of MHS, \cite[Lem.~4.5]{hain:torelli}, the monodromy homomorphism extends to a homomorphism
$$
\pi_1(\MHS) \ltimes \cG_{g,n+\uu} \to \Aut \Q\pi_1(X,\vv_o)^\wedge.
$$
Denote its image by $\cGhat_{g,n+\uu}$ and the image of $\cG_{g,n+\uu}$ by $\cGbar_{g,n+\uu}$. It is normal in $\cGhat_{g,n+\uu}$. The group $\cGhat_{g,n+\uu}$ is an extension
$$
1 \to \Uhat_{g,n+\uu} \to \cGhat_{g,n+\uu} \to \GSp(H) \to 1,
$$
where $\GSp$ denotes the general symplectic group and $\Uhat_{g,n+\uu}$ is prounipotent.\footnote{One can argue as in \cite{hain-matsumoto} that, if $g\ge 3$, then then $\U^\MT_{X,\vv} \to \Uhat_{g,n+\uu}$ is an isomorphism if and only if $\pi_1(\MHS) \to \GSp(H)$ is surjective; the Griffiths invariant $\nu(\Xbar) \in \Ext^1_\MHS(\Q,PH^3(\Jac \Xbar(2)))$ of the Ceresa cycle in $\Jac\Xbar$ is non-zero; and if the points $\kappa_j := (2g-2)x_j - K_\Xbar \in (\Jac \Xbar)\otimes \Q$, $0\le j \le n$, are linearly independent over $\Q$. This holds for general $(\Xbar,D,\vv)$.}

\begin{proposition}
For each complex structure $\phi : (\Sbar,P,\vv_o) \to (\Xbar,D,\vv)$, the coordinate ring $\cO(\cGhat_{g,n+\uu}/\cGbar_{g,n+\uu})$ has a canonical MHS. These form an admissible variation of MHS over $\M_{g,n+\uu}$ with trivial monodromy. Consequently, the MHS on $\cO(\Uhat_{g,n+\uu}/\Ubar_{g,n+\uu})$ does not depend on the complex structure $\phi$.
\end{proposition}

\begin{proof}
The first task is to show that the $\cGhat_{g,n+\uu}$ form a local system over $\M_{g,n+\uu}$. This is not immediately clear, as the size of the Mumford--Tate group depends non-trivially on complex structure on $(\Sbar,P,\vv_o)$. To this end, let $x=(\Xbar,D,\vv)$ be a point of $\M_{g,n+\uu}$. Denote the relative completion of $\pi_1(\M_{g,n+\uu},x)$ by $\cG_x$. Let $y=(\Ybar,E,\vv')$ be another point of $\M_{g,n+\uu}$ and let $\cG_y$ be the relative completion of $\pi_1(\M_{g,n+\uu},y)$. Denote the relative completion of the torsor of paths in $\M_{g,n+\uu}$ from $x$ to $y$ by $\cG_{x,y}$. Its coordinate ring has a natural MHS and the multiplication map
$$
\cG_x \times \cG_{x,y} \to \cG_y
$$
is a morphism of MHS \cite{hain:malcev}. This is equivalent to the statement that the map
$$
(\pi_1(\MHS)\ltimes \cG_x) \times \cG_{x,y} \to \pi_1(\MHS)\ltimes \cG_y
$$
defined by $(\sigma,\lambda,\gamma) \mapsto (\sigma, \gamma^{-1}\lambda \gamma)$ is a $\pi_1(\MHS)$-equivariant surjection, where $\alpha \in \pi_1(\MHS)$ acts by
$$
\alpha : (\sigma,\lambda,\gamma) \mapsto (\alpha\sigma\alpha^{-1},\alpha\cdot\lambda,\alpha\cdot\gamma) \text{ and } \alpha : (\sigma,\mu) \mapsto (\alpha\sigma\alpha^{-1},\alpha\cdot\mu) .
$$
The diagram
$$
\xymatrix{
(\pi_1(\MHS)\ltimes \cG_x) \times \cG_{x,y} \ar[r] \ar[d] & \pi_1(\MHS)\ltimes \cG_y \ar[d] \cr
\Aut\Q\pi_1(X,\vv)^\wedge \times \cG_{x,y} \ar[r] & \Aut\Q\pi_1(Y,\xi')^\wedge
}
$$
commutes, where $Y=\Ybar-E$ and where the bottom arrow is induced by parallel transport in the local system whose fiber over $x$ is $\Aut\Q\pi_1(X,\vv)^\wedge$. This implies that there is a morphism $\cGhat_x \times \cG_{x,y} \to \cGhat_y$ that is compatible with path multiplication. It follows that the $\cG_x$ form a local system over $\M_{g,n+\uu}$.

We now prove the remaining assertions. The monodromy action of $\G_{g,n+\uu}$ on $\cGhat_{g,n+\uu}/\cGbar_{g,n+\uu}$ is the composite
$$
\G_{g,n+\uu} \to \cG_{g,n+\uu}(\Q) \to \Aut\big(\cGhat_{g,n+\uu}/\cGbar_{g,n+\uu}\big)(\Q),
$$
where the first homomorphism is the canonical map, and the second is induced by conjugation. It is easily seen to be trivial as $\cGbar_{g,n+\uu}$ is normal in $\cGhat_{g,n+\uu}$.

The coordinate ring of $\cGhat_{g,n+\uu}/\cGbar_{g,n+\uu}$ has a MHS as the inclusion $\cGbar_{g,n+\uu} \to \cGhat_{g,n+\uu}$ is $\pi_1(\MHS)$-equivariant. This variation has no geometric monodromy, and so is constant by the theorem of the fixed part. Since $\Ubar_{g,n+\uu} = \Uhat_{g,n+\uu} \cap \cGbar_{g,n+\uu}$, the map
$$
\Uhat_{g,n+\uu}/\Ubar_{g,n+\uu} \to \cGhat_{g,n+\uu}/\cGbar_{g,n+\uu}
$$
is a $\pi_1(\MHS)$-equivariant inclusion. It follows that $\Uhat_{g,n+\uu}/\Ubar_{g,n+\uu}$ is also a constant variation of MHS over $\M_{g,n+\uu}$.
\end{proof}

The homomorphism $\tautilde_\dd : \cG_{g,n+\uu} \to \Sp(H)\ltimes H$ lifts to a homomorphism
$$
\tautilde_\dd : \cGhat_{g,n+\uu} \to \GSp(H)\ltimes H
$$
Its kernel is the group $\Uhat_{g,n+\uu}^\dd$ defined in the previous section. Since $\Ubar_{g,n+\uu}^\dd = \Uhat_{g,n+\uu}^\dd \cap \Ubar_{g,n+\uu}$, we have:

\begin{corollary}
$\cO(\Uhat_{g,n+\uu}^\dd/\Ubar_{g,n+\uu}^\dd)$ is a constant VMHS over $\M_{g,n+\uu}$.
\end{corollary}

\begin{proposition}
\label{prop:surjection}
Assuming hypothesis (\ref{hypoth:ihara}), there is a canonical surjection $\pi_1(\MTM) \to \cGhat_{g,n+\uu}/\cGbar_{g,n+\uu}$.
\end{proposition}

The prounipotent analogue of the proof of Oda's Conjecture \cite{takao} should imply that this is an isomorphism.

\begin{proof}[Sketch of Proof]
Since the variation $\cO(\Uhat_{g,n+\uu}/\Ubar_{g,n+\uu})$ is constant, it extends over the boundary of $\M_{g,n+1}$. Since the variation of MHS over $\M_{g,n+\uu}$ with fiber $\u_{g,n+\uu}$ is admissible, it has a limit MHS at each tangent vector of the boundary divisor $\Delta$ of $\Mbar_{g,n+\uu}$. These tangent vectors correspond to first order smoothings of an $(n+1)$-pointed stable nodal curve of genus $g$ together with a tangent vector at the initial point $x_0$. For each such maximally degenerate stable curve\footnote{These correspond to pants decompositions of $(\Sbar,P,\vv)$ and also to the 0-dimensional strata of $\Mbar_{g,n+\uu}$.} $(\Xbar_0, P,\vv_0)$ of type $(g,n+\uu)$, Ihara and Nakamura \cite{ihara-nakamura} construct a proper flat curve
$$
\X \to \Spec\Z[[q_1,\dots,q_N]],\qquad N = \dim \Mbar_{g,n+1} =  3g + n - 2
$$
with sections $x_j$, $0\le j \le n$ and a non-zero tangent vector field $\vv$ along $x_0$ that specialize to the points of $P$ and the tangent vector $\xi_0$ at $q=0$. The projection is smooth away from the divisor $q_1 q_2 \dots q_N = 0$. These are higher genus generalizations of the Tate curve in genus 1.

There is a limit MHS on each of
$$
\Q\pi_1(X_\qq,\vv)^\wedge,\ \cO(\Uhat_{g,n+\uu}),\ \cO(\Ubar_{g,n+\uu})
$$
corresponding to the tangent vector $\qq := \sum_{j=1}^N \partial/\partial q_j$ of $\Mbar_{g,n+\uu}$ at the point corresponding to $(\Xbar_0,P,\vv_0)$. These can be thought of as MHSs on the invariants of $(X_\qq,\vv)$, where $\Xbar_\qq$ denotes the fiber of $\X$ over $\qq$ and $X_\qq$ the corresponding affine curve.  

The hypothesis (\ref{hypoth:ihara}) --- which we claim is proved in \cite{hain:ihara} --- is that these MHS are Hodge realizations of objects of $\MTM$. This implies that each has an action of $\pi_1(\MTM)$ and that the action of $\pi_1(\MHS)$ on each factors through the canonical surjection $\pi_1(\MHS) \to \pi_1(\MTM)$. Brown's result \cite{brown} asserts that $\pi_1(\MTM)$ acts faithfully on
$$
\Q\pi_1(\Pminus,\vv_o)^\wedge.
$$
This implies that it also acts faithfully on $\Q\pi_1(X_\qq,\vv)^\wedge$ as (by the construction in \cite{hain:ihara}), the unipotent path torsor of $X_\qq$ is built up from the path torsors of copies of $\Pminus$ (and is 6 canonical tangent vectors) in $\Xbar_\qq$. In other words, $\MT_{X_\qq,\vv}$ is naturally isomorphic to $\pi_1(\MTM)$. This implies that the homomorphism $\pi_1(\MHS) \to \cGhat_{g,n+\uu}/\cGbar_{g,n+\uu}$ induces a surjective homomorphism $h:\pi_1(\MTM) \to \cGhat_{g,n+\uu}/\cGbar_{g,n+\uu}$.
\end{proof}

Recall from the introduction that $\K$ is the prounipotent radical of $\pi_1(\MTM,\w^B)$.

\begin{corollary}
Assuming hypothesis (\ref{hypoth:ihara}), there is a canonical surjection $\K \to \Uhat_{g,n+\uu}^\dd/\Ubar_{g,n+\uu}^\dd$.
\end{corollary}

\end{document}